\RequirePackage{fix-cm}
\documentclass[smallextended]{svjour3}       
\smartqed  
\usepackage{graphicx}
\usepackage{etex,environ,etoolbox}
\usepackage{algorithm}
\usepackage{algorithmic}
\usepackage[english]{babel}
\usepackage[utf8]{inputenc}  
\usepackage[T1]{fontenc} 
\usepackage{amsmath}

\usepackage{amsthm}

\usepackage{pgfplots}
\usepackage{graphicx}
\usepackage{multicol}
\usepackage{geometry}

\usepackage{tikz}
\usepackage{booktabs}
\usepackage{algorithm}

\usepackage{array}
\usepackage{bm}
\usepackage{dsfont}
\usepackage{amssymb}
\usepackage{listings}

\newcommand{\troppoly}[1]{#1}
\DeclareMathOperator*{\argmax}{\arg\,\max}
\DeclareMathOperator*{\argmin}{\arg\,\min}
\newcommand{\R}{\mathbb{R}} 
\newcommand{\Rmax}{\mathbb{R}_{\max}}

\newcommand{\Z}{\mathbb{Z}}         

\newcommand{\N}{\mathbb{N}} 
\newcommand{\Rinfty}{\R\cup\{-\infty\}}

\newcommand{\dom}{\operatorname{dom}}

\newcommand{\TP}{\mathbb{T} \mathbb{P}}
\newcommand{\tplus}{\oplus}  

\newcommand{\tdot}{\odot}

\newcommand{\cro}[1]{\left[ #1 \right]}
\makeatletter
\newcommand{\raisemath}[1]{\mathpalette{\raisem@th{#1}}}
\newcommand{\raisem@th}[3]{\raisebox{#1}{$#2#3$}}
\makeatother

\usepackage{proofsatend}

\newtoggle{moveproofsatend}
\toggletrue{moveproofsatend}
\togglefalse{moveproofsatend}

\newtoggle{shortversion}
\newtoggle{veryshortversion}
\togglefalse{shortversion}
\togglefalse{moveproofsatend}

\newenvironment{movableproof}{}{}
\iftoggle{moveproofsatend}{

}{

}

\begin{document}

\title{A bilevel optimization model for load balancing in mobile networks through price incentives}


\author{Marianne Akian         \and
        Mustapha Bouhtou  \and
        Jean Bernard Eytard \and
        St\'ephane Gaubert
}


\institute{M. Akian \at
              INRIA, CMAP, Ecole Polytechnique, CNRS \\
              Route de Saclay \\
              91128 Palaiseau Cedex, France \\
              Tel.: +331 69 33 46 39\\
              \email{marianne.akian@inria.fr}          
           \and
           M. Bouhtou \at
              Orange Labs \\
              44, avenue de la R\'epublique \\
              92320 Chatillon, France \\
              \email{mustapha.bouhtou@orange.com}
           \and
           J.B. Eytard \at
              INRIA, CMAP, Ecole Polytechnique, CNRS \\
              Route de Saclay \\
              91128 Palaiseau Cedex, France \\
              \email{jean-bernard.eytard@inria.fr} 
           \and
           S. Gaubert \at
              INRIA, CMAP, Ecole Polytechnique, CNRS \\
              Route de Saclay \\
              91128 Palaiseau Cedex, France \\
              Tel.: +331 69 33 46 03\\
              \email{stephane.gaubert@inria.fr} \\            
           \vfill
           \textit{Preprint submitted to EURO Journal on Computational Optimization \hfill October 31, 2018}        
}

\date{}

\maketitle

\begin{abstract}
We propose a model of incentives for data pricing in large mobile networks,
in which an operator wishes to balance the number of connections (active users) of
different classes of users
in the different cells and at different time instants, in order to
ensure them a sufficient quality of service.
We assume that each user has a given total demand per day for different
types of applications,
which he may assign to
different time slots and locations, depending on his own
mobility, on his preferences and on price discounts proposed by the
operator.
We show that this can be cast as a bilevel programming problem with a
special
structure allowing us to develop a polynomial time decomposition
algorithm suitable for large networks.
First, we determine the optimal number of connections (which maximizes a
measure of balance); next, we solve an inverse problem and determine the
prices generating this traffic. Our results exploit a recently developed
application
of tropical geometry methods to mixed auction problems, as well as
algorithms in discrete convexity (minimization of discrete convex
functions in the sense of Murota).
We finally present an application on real data provided by Orange
and we show the efficiency of the model to reduce the peaks of congestion.
\end{abstract}

\section{Introduction}
With the development of new mobile data technologies (3G, 4G), the demand for using the Internet with mobile phones has increased rapidly. Mobile service providers (MSP) have to confront congestion problems in order to guarantee a sufficient quality of service (QoS). 

Several approaches have been developed to improve the quality of service, coming from different fields of the telecommunication engineering and economics. For instance, one can refer to Bonald and Feuillet \cite{bonald2013network} for some models of performance analysis to optimize the network in order to improve the QoS. 
One of the promising alternatives to solve such problems consists in using efficient pricing schemes in order to encourage customers to shift their mobile data consumption. In \cite{maille2006pricing}, Maill\'e and Tuffin describe a mechanism of auctions based on game-theoretic methods for pricing an Internet network, see also \cite{maille2014telecommunication}. In \cite{altman2006pricing}, Altman et al.\ study how to price different services by using a noncooperative game. These different approaches are based on congestion games. In the present work, we are interested in how a MSP can improve the QoS by balancing the traffic in the network. We wish to determine in which locations, and at which time instants, it is relevant to propose price incentives, 
and to evaluate the influence of these incentives on the quality of service.

This kind of problem belongs to smart data pricing. We refer the reader to the survey of Sen et al.~\cite{sen2013survey} and also to the collection of articles \cite{sen2014smart}. Finding efficient pricing schemes is a revenue management issue. The first approach consists in usage-based pricing; the prices are fixed monthly by analysing the use of the former months. It is possible to improve this scheme by identifying peak hours and non-peak hours and proposing incentives in non-peak hours in order to decrease the demand at peak hours and to better use the network capacity at non-peak hours. 
This leads to time-dependent pricing. Such a scheme for mobile data is developed by Ha et al.\ in \cite{ha2012tube}. The prices are determined at different time slots and based on the usage of the previous day in order to maximize the utility of the customers and the revenue of the MSP. 
This pricing scheme was concretely implemented by AT\&T, showing the relevance of such a model. 
In another approach, Tadrous et al.\ propose a model in which the MSP anticipates peak hours and determines incentives for proactive downloads \cite{tadrous2013pricing}. 

The latter models concern only the time aspects. One must also take into account the spatial aspect in order to optimize the demand between the different locations. In \cite{ma2014time}, Ma, Liu and Huang present a model depending on time and location of the customers where the MSP proposes prices and optimizes his profit taking into account the utility of the customers. 

Here, we assume (as in~\cite{ma2014time}) that the MSP proposes incentives at different time and places. Then, customers optimize their data consumption by knowing these incentives and the MSP optimizes a measure of the QoS. In this way, we introduce a bilevel model in which the provider proposes incentives in order to balance the traffic in the network and to avoid as much as possible the congestion (high level problem), and customers optimize their own consumption for the given incentives (low level problem). 


Bilevel programs have been widely studied, see the surveys of Colson, Marcotte and Savard \cite{colson2007overview} and of Dempe \cite{dempe2003bilevel}. They represent an important class of pricing problems in the sense that they model a leader wanting to maximize his profit and proposing prices to some followers who maximize themselves their own utility. Most classes of bilevel programs are known to be NP-hard. Several methods have been introduced to solve such problems. For instance, if the low level program is convex, 
it can be replaced by its Karush-Kuhn-Tucker optimality conditions and the bilevel problem becomes a classical one-stage optimization problem, which is however generally non convex. If some variables are binary or discrete, and the objective function is linear, the global bilevel problem can be rewritten as a mixed integer program, as in Brotcorne et al.~\cite{brotcorne2000bilevel}. 

In the present work, we optimize the consumption of each customer in a large area (large urban agglomeration) during typically one day divided in time slots of one hour, taking into account the different types of customers and of applications that they use. 
Therefore, we have to confront both with the difficulties
inherent to bilevel programming and with the large number of variables (around $10^7$).
Hence, we need to find polynomial time algorithms, or fast
approximate methods, for classes of problems of a very large scale, which,
if treated directly, would lead to mixed integer linear or nonlinear programming formulations beyond the capacities of current off-the-shelve solvers.

This motivated us to introduce a different approach, based on tropical geometry. 
Tropical geometry methods have been recently applied by Baldwin and Klemperer in \cite{baldwin2012tropical} to an auction problem.
This has been 
further developed by Yu and Tran \cite{tran2015product}.
In these approaches,
the response of an agent to a price is represented by a certain
polyhedral complex (arrangement of tropical hypersurfaces). This
approach is intuitive since it allows one to
vizualize geometrically the behavior of the agents: each cell
of the complex corresponds to the set of incentives leading
to a given response. Then, we vizualize the collective response of
a group of customers by ``superposing'' (refining) the polyhedral
complexes attached to every customer in this group. 
We apply here this idea
to represent the response of the low-level optimizers in a bilevel
problem. This leads to the following decomposition method: first we compute,
among all the admissible consumptions of the customers, the one which 
maximizes a measure of balance of the network; then, we determine
the price incentive which achieves this consumption. 
In this way, a bilevel
problem is reduced to the minimization of a convex function over a certain Minkowski sum of sets. We identify situations in which the latter problem can be solved in polynomial time, by exploiting the discrete convexity results
developed by Murota~\cite{murota2003discrete}. 
In this approach, a critical step is to check the membership of a vector
to a certain Minkowski sum of sets of integer points of polytopes. 
In our present model, these polytopes, which represent the possible consumptions of one customer, have a
remarkable combinatorial structure
(they are hypersimplices). Exploiting this combinatorial structure,
we show that this critical
step can be performed quickly, by reduction
to a shortest path problem in a graph. This leads
to an exact solution method when there is only one type
of contract and one type of application sensitive to price
incentive, and to a fast approximate method in the general case.

We finally present the application of this model on real data from Orange and show how price incentives can improve the QoS by balancing the number of active customers in an urban agglomeration during one day.
These results indicate that a price incentive mechanism can effectively
improve the satisfaction of the users by displacing their consumption
from the most loaded regions of the space-time domain to less loaded
regions. 



The paper is organized as follows. 
In Section~\ref{sec-model}, we present the bilevel model. In Section~\ref{sec-polytime}, we explain how a certain polyhedral complex can be used to represent the user's responses, and we describe the decomposition method.
In Section~\ref{sec-first}, we deal with the high level problem and identify special cases which are solvable in polynomial time.
In Section~\ref{subsec2-4}, we develop accelerated algorithms which enable to solve bilevel problems with a large number of customers.
In Section~\ref{sec-general}, we propose a general relaxation method. 
The application to the instance provided by Orange is presented in Section~\ref{sec-exp}.

The first results of this article (without proofs) were published in the proceedings of the conference WiOpt 2017~\cite{eytard2017bilevel}.  

\section{A bilevel model} \label{sec-model}


We consider a time horizon of one day, divided in $T$ time slots numbered $t \in \cro{T}= \lbrace 1, \dots T \rbrace$, and a network divided in $L$ different cells numbered $l \in \cro{L}$.  We assume that $K$ customers, numbered $k \in \cro{K}$, are in the network. The customers have different types of contracts $b \in \cro{B}$ and they make requests for different types of applications $a \in \cro{A}$ (web/mail, streaming, download, \dots). We denote by $\mathcal{K}^{b}$ the set of customers with the contract $b$. A given customer $k \in \mathcal{K}^{b}$ is characterized by the following data.
We denote by $L_{t}^{k} \in \left[ L \right] $ the position of the customer $k$ at each time $t\in \cro{T}$, so that the sequence $(L^{k}_1,\ldots, L^{k}_T)$ represents the trajectory of this customer. We assume that this trajectory is deterministic, so we consider customers with a regular daily mobility (for example, the trip between home and work).
We denote by $\rho_{k}^{a}(t)$ the inclination of a customer $k$ to make a request for an application of type $a$ at time $t \in \cro{T}$.
We suppose that customer $k$ wishes to make a fixed number of requests $R_{k}^{a}\leq T$ using the application $a$ during the day. 
We consider a set of time slots $\mathcal{I}_{k}^{a} \subset \cro{T}$  in which the customer $k$ decides not to consume the application $a$.

We denote by $u_{k}^{a}(t)$ the consumption of the customer $k$ for the application $a$ at time $t$, 
setting $u_{k}^{a}(t)=1$ if $k$ is active at time $t$ and makes a request of type $a$ and $u_{k}^{a}(t)=0$ otherwise. 
Therefore, the number $N^{a,b}(t,l)$ of active customers with contract $b$ for the application $a$ at time $t$ and location $l$ is given by
\(
N^{a,b}(t,l)=\sum_{k \in \mathcal{K}^{b}}{u_{k}^{a}(t)\mathds{1}(L_{t}^{k}=l)}
\),
where $\mathds{1}$ denotes the indicator function, and the total number of active customers $N(t,l)$ at time $t$ and location $l$ is given by $N(t,l)=\sum_{a}{\sum_{b}{N^{a,b}(t,l)}}$.

We consider the following two-stage model of price incentives. The first stage consists for the operator in announcing a discount $y^{a,b}(t,l)$ at time $t$ and location $l$ for the customers of contract $b$ making requests of type $a$. We consider only nonnegative discounts, so $y^{a,b}(t,l) \geq 0$. The second stage models the behavior of customers who modify their consumption by taking the discounts into account. We will assume the preference of a customer $k$ of contract $b$ for consuming at time $t$ becomes $\rho_{k}^{a}(t)+\alpha_{k}^{a} y^{a,b}(t,L_{t}^{k})$, where $\alpha_{k}^{a}$ denotes the sensitivity of customer $k$ to price incentives for the application $a$.
It corresponds to classical linear utility functions, see e.g. \cite{baldwin2012tropical}. 
We also assume that the customers cannot make more than one request at each time, that is $\forall t \in \cro{T}$, $\sum_{a}{u_{k}^{a}(t)} \leq 1$. Therefore, each customer $k$ determines his consumptions $u_{k}^{a}=(u_k^{a}(t))_{t \in \cro{T}}\in \{0,1\}^T$ for the applications, as an optimal solution of the linear program:
\begin{problem}[Low-level, customers]\label{lowlevel}
\begin{equation} 
\max_{u_{k}^{a}\in \{0,1\}^T}  \sum_{a \in \cro{A}}{\sum_{t=1}^{T}{\left[ \rho_{k}^{a}(t)+\alpha_{k}^{a} y^{a,b}(t,L_{t}^{k}) \right] u_{k}^{a}(t)}} 
\end{equation}
\begin{align*}
& \text{s.t. } \; \forall a \in \cro{A}, \sum_{t=1}^{T}{u_{k}^{a}(t)}=R_{k}^{a},  \quad \forall t \in \cro{T}, \sum_{a \in \cro{A}}{u_{k}^{a}(t)} \leq 1 \enspace \\
& \forall t \in \mathcal{I}_{k}^{a}, \forall a \in \cro{A}, u_{k}^{a}(t)=0
\enspace .
\end{align*}
\end{problem}
 
Consequently, each price $y^{a,b}=(y^{a,b}(t,l))_{t\in \cro{T},\; l\in\cro{L}}$ determines the possible individual consumptions $u_{k}^{a}$ for the users with contract $b$, and so the possible cumulated traffic vectors $N^{a,b}=(N^{a,b}(t,l))_{t\in \cro{T},\; l\in\cro{L}}$ and $N=\sum_{a}{\sum_{b}{N^{a,b}}}$.  The aim of the operator is, through price incentives, to balance the load in the network into the different locations and time slots to improve the quality of service perceived by each customer. We introduce a coefficient $\gamma_{b}$ relative to the kind of contracts of the different customers in order to favor some classes of premium customers. 
In \cite{lee2005non}, Lee et al.\ suppose that the satisfaction of a customer depends on his perceived throughput, which can be considered as inversely proportional to the number of customers in the cell.
Here, we assume that the satisfaction of each customer $k$ in the cell $l \in \cro{L}$ is a nonincreasing function $s_{l}^{a,b}$ of the total number of active customers in the cell $N(t,l)$, depending on the characteristics of the cell, of the type of application the user wants to do (some applications like streaming need a higher rate than others) and on the type of contract. 
We also assume that the satisfaction of all the customers with contract $b$ using
a given application $a$ in a given cell
is maximal until the number of active customers
reaches a certain threshold $N_l^{a,b}$, then $s_{l}^{a,b}(N(t,l))=1$ for $N(t,l)\leq N_l^{a,b}$. After this threshold, the satisfaction decreases until a critical value $N_{l}^{C}$. We add the constraint $\forall t\in \cro{T},\; \forall l \in \cro{L}, \; N(t,l) \leq N_{l}^{C}$ to prevent the congestion. For non-real time services like web, mail, download, the satisfaction function can be viewed as a concave function of the throughput, like $1-e ^{-\delta/\delta_{c}}$ where $\delta$ denotes the throughput, see Moety et al.~\cite{moety2016satisfaction}. Hence, we will consider that for contents like web, mail and download, $N_l^{a,b}=N_l^1$, $s_{l}^{a,b}(n)=1$ for $n \leq N_{l}^{1}$ and $s_{l}^{a,b}(n)=1-\lambda_{b} \exp\left(- \frac{2N_{l}^{C}}{n-N_{l}^{1}}\right)$ for $N_{l}^{1} \leq n \leq N_{l}^{C}$ where $\lambda_{b}$ is a positive parameter depending on the kind of contract of the customer. The more expensive the contract of the customer is, the larger is $\lambda_{b}$. We can prove that this function is concave for $0 \leq n \leq  N_{l}^{C}$. For real time services like video streaming, the customers need a more important throughput to ensure a good QoS \cite{lee2005non}. We will here consider the same type of functions $s_{l}^{a,b}$ but with $N_l^1$ replaced by $N_l^{a,b}=0$, that is $s_{l}^{a,b}(n)=1-\lambda_{b} \exp\left(- \frac{2N_{l}^{C}}{n}\right)$ for $0 < n \leq N_{l}^{C}$.

\begin{figure}[h]
\begin{center}
\begin{tikzpicture} [scale=2]
	\draw [red] (0,1.5)--(1.6,1.5);
	\draw (1.6,0)--(1.6,1.5);
    \draw [blue, domain=0.01:5,variable=\x] plot({\x},{1.5*(1-exp(2-10/\x))});
    \draw [blue,dashed, domain=0.01:5,variable=\x] plot({\x},{1.5*(1-2/3*exp(2-10/\x))});
    \draw [red, domain=1.61:5,variable=\x] plot({\x},{1.5*(1-exp(2-10/(\x-1.6)))});
    \draw [red,dashed, domain=1.61:5,variable=\x] plot({\x},{1.5*(1-2/3*exp(2-10/(\x-1.6)))});
    \draw [->](0,0)--(5,0) node[right]{$N(t,l)$};
    \draw [->](0,0)--(0,2) node[left]{$s_{l}^{a,b}$};
    \draw (0,1.5) node[left]{1};
    \draw (0,0) node[left]{0};
    \draw (1.6,0) node[below]{$N_{l}^{1}$};
    \draw (5,0) node[below]{$N_{l}^{C}$};
\end{tikzpicture}
\caption{Different kind of satisfaction functions of the number of active customers in a cell. The blue ones are those for streaming contents whereas the red ones are those for web, mail and download contents. The dashed ones corresponds to the satisfaction of standard customers, the continuous ones to the satisfaction of premium customers.}
\label{satisfaction}
\end{center}
\end{figure}
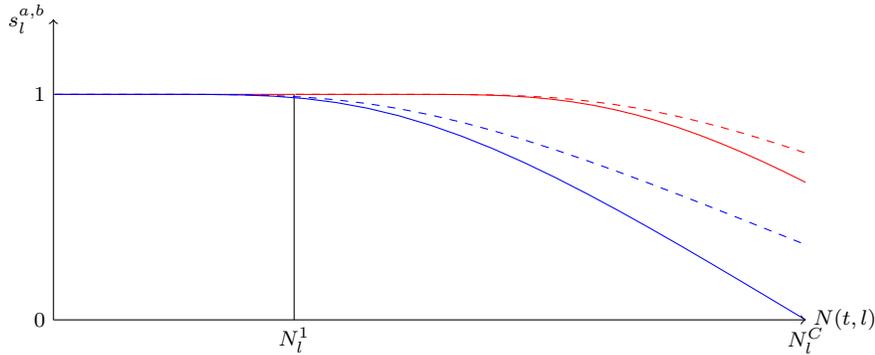

%
%

So, the first stage consists in maximizing the global satisfaction function $s$ which depends on the vectors $N^{a,b}\in \N^{T\times L}$ and is defined by:

\begin{align*}
&s(N^{a,b}) = \sum_{t=1}^{T}\sum_{a \in \cro{A}}\sum_{b \in \cro{B}}\sum_{k \in \mathcal{K}^{b}} \gamma_{b}s_{L_{t}^{k}}^{a,b}(N(t,L_{t}^{k}))u_{k}^{a}(t) \\
& =\sum_{t=1}^{T}\sum_{a \in \cro{A}}\sum_{b \in \cro{B}}\sum_{k \in \mathcal{K}^{b}}\sum_{l=1}^{L} \gamma_{b}s_{l}^{a,b}(N(t,l))\mathds{1}(L_{t}^{k}=l)u_{k}^{a}(t) \\
& = \sum_{t=1}^{T}\sum_{l=1}^{L}\sum_{a \in \cro{A}}\sum_{b \in \cro{B}} \gamma_{b}N^{a,b}(t,l) s_{l}^{a,b}(N(t,l))
\end{align*} 
with $\forall b \in \cro{B}, \gamma_{b}>0$. Our final model consists in solving the following bilevel program: 
\begin{problem}[High-level, provider] \label{highlevel}
\begin{equation}
 \max_{y^{a,b} \in \R_{+}^{T \times L}}  \sum_{t=1}^{T}{\sum_{l=1}^{L}{\sum_{a \in \cro{A}}{\sum_{b \in \cro{B}}{\gamma_{b} N^{a,b}(t,l) s_{l}^{a,b}(N(t,l))}}}} 
\end{equation}
where  $\forall t \in \cro{T},\; l \in \cro{L}, \; N(t,l)=\sum_{a=1}^{A}{\sum_{b=1}^{B}{N^{a,b}(t,l)}}$, and $N(t,l) \leq N_{l}^{C}$, $\forall t \in \cro{T},\; l \in\cro{L},\; a\in\cro{A},\; b\in \cro{B}, \; N^{a,b}(t,l)=\sum_{k\in \mathcal{K}^{b}}{u_{k}^{a}(t)\mathds{1}(L_{t}^{k}=l)}$, and $\forall k \in \cro{K}$, the vectors $u_{k}^{a}$ are solutions of Problem~\ref{lowlevel}.
\end{problem}

\section{A decomposition approach for solving the first model} \label{sec-polytime}

We will present a decomposition method for solving the previous bilevel problem. 
In this section, and in the next two ones, we suppose that there is only one kind of application and one kind of contract. This special case is already relevant in applications: it covers the case when, for instance, only the download requests are influenced by price incentives, whereas other requests like streaming or web are fixed. Whereas the analytical results of the present section carry over to the general model, the results of the next two sections (polynomial time solvability) are only valid under these restrictive assumptions.
We shall return to the general case in Section~\ref{sec-general}, developing a fast approximate algorithm for the general model based on the present principles.

In the above special case, the bilevel model can be rewritten:
\begin{equation*} \label{highlevel2}
 \max_{y \in \R_{+}^{T \times L}}  \sum_{t=1}^{T}\sum_{l=1}^{L} N(t,l) s_{l}(N(t,l)) 
\end{equation*}
where $\forall t, l \; N(t,l) \leq N_l^C$ and $N(t,l)=\sum_{k\in \cro{K}}{u_k^*(t)\mathds{1}(L_{t}^{k}=l)}$, and for each $k \in \cro{K}$ the vectors $u_k^*$ are solutions of the problem:
\begin{equation*} \label{lowlevel2}
\max_{u_{k}\in \{0,1\}^T}  \sum_{t=1}^{T}{\left[ \rho_{k}(t)+\alpha_{k} y(t,L_{t}^{k}) \right] u_{k}(t)}
\end{equation*}
\[
\text{s.t. } \sum_{t=1}^{T}{u_{k}(t)}=R_{k}, \quad \enspace 
\forall t \in \mathcal{I}_{k}, u_{k}(t)=0, \quad  \enspace 
\]

In order to deal more abstractly with the bilevel model, we introduce the notation $u_{k}(t,l)=u_{k}(t)\mathds{1}(L_{t}^{k}=l)$.
Hence, we have $u_{k}(t,l)=0$ if $L_{t}^{k} \neq l$. By defining the set $\mathcal{J}_{k}=\{ (t,l) \; \mid \; t \in \mathcal{I}_{k}$ or $L_{t}^{k} \neq l \}$,
we have that $(t,l) \in \mathcal{J}_k$ implies that $u_k(t,l)=0$. 
We can then define 
 $\rho_k(t,l)=\rho_k(t)/ \alpha_{k}$ if $(t,l) \notin \mathcal{J}_k$ and $\rho_k(t,l)=-\infty$ otherwise. 
Then, we can rewrite each low-level problem as:
\begin{equation*} \label{lowlevel2bis}
\max_{u_k \in F_k} \sum_{t,l} \left[ \rho_k(t,l) + y(t,l) \right] u_k(t,l)
\end{equation*}
where $F_k=\lbrace u \in \{ 0,1 \}^{T\times L} \mid \sum_{t,l} u(t,l) =R_k \; \text{and}  \; \forall (t,l) \in \mathcal{J}_k, u(t,l)=0 \rbrace$, 
and the global bilevel problem becomes:
\begin{equation*} \label{highlevel2bis}
\max_{y \in \R_{+}^{T \times L}} \sum_{t,l} f_{l}(N(t,l)) \quad 
\text{s.t.} \;  \forall (t,l)\quad N(t,l) \leq N_l^C, \quad N(t,l)=\sum_{k=1}^{K}{u_{k}(t,l)}  
\end{equation*}
with $f_l : x \in \R_+ \mapsto x s_l(x)$.
Notice that the set $\mathcal{J}_k$ corresponds to the set of couples $(t,l)$ such that $\rho_k(t,l)=-\infty$. 
It is possible to enumerate all the couples $(t,l) \in \cro{T} \times \cro{L}$.
Let us define $n= T \times L$ and associate each couple $(t,l)$ to an integer $i \in \cro{n}$. 
The quantities $\rho_k(t,l)$, $u_k(t,l)$, $N(t,l)$ and $y(t,l)$ can be respectively denoted by $\rho_k(i)$, $u_k(i)$, $N_i$ and $y_i$.
The function $f_l$ and the integer $N_l^C$ can be respectively denoted by $f_i$ and $N_i^C$. 
It means that for two indices $i$ and $j$ associated to two couples $(t,l)$ and $(t',l)$ with the same $l$, we have $f_i=f_j := f_l$ and $N_i^C=N_j^C := N_l^C$.
The low-level problem can be rewritten:
\begin{problem}[Abstract low-level problem] \label{lowlevel3}
\begin{equation} 
\max_{\scriptstyle u_{k}\in F_{k}}  \sum_{i=1}^n \left[ \rho_{k}(i)+ y_i \right] u_{k}(i)
\end{equation}
where $F_{k}=\lbrace u \in \lbrace 0,1 \rbrace ^n | \sum_{i=1}^n u(i)=R_{k} $ and $\forall i \in \mathcal{J}_{k}, u(i)=0 \rbrace$. 
\end{problem}
The global bilevel problem is:
\begin{problem}[Bilevel problem] \label{highlevel3}
\begin{equation} 
\max_{y \in \R_+^n} \sum_{i=1}^n f_i(N_i) \quad 
\text{s.t.} \;  \forall i, N_i \leq N_i^C, \quad N_i=\sum_{k=1}^{K}{u_k^*(i)}   
\end{equation}
with for all $k \in \cro{K}$, $u_k^*$ solution of Problem \ref{lowlevel3}.
\end{problem}

\begin{lemma}
Suppose that the functions $s_{i}$ are nonincreasing and concave on $\left[ 0, N_{i}^{C} \right]$. Then, the functions $f_{i}$ are also concave on $\left[ 0, N_{i}^{C} \right]$.
\end{lemma}

\begin{proof}
The result comes easily if we suppose that the functions $s_{i}$ are twice differentiable, because we have: 
\[
\forall x \in \left[ 0, N_{i}^{C} \right], f_{i}''(x)=xs_{i}''(x)+2s_{i}'(x) \leq 0
\enspace .\]

We could deduce that the same is true without the differentiability assumption by a density argument, writing a concave function as a pointwise limit of smooth concave functions. However, we prefer to provide the following elementary argument.
Consider $0 \leq x \leq y \leq N_{i}^{C}$ and $t \in \left[ 0,1 \right]$. Because $s_{i}$ is nonincreasing, we have $s_{i}(x) \geq s_{i}(y)$. We have:
\begin{align*}
tf_{i}(x)+(1-t)f_{i}(y) &= txs_{i}(x)+(1-t)ys_{i}(y) \\
&= (tx+(1-t)y) \left[ \dfrac{tx}{tx+(1-t)y}s_{i}(x)+ \dfrac{(1-t)y}{tx+(1-t)y}s_{i}(y) \right] \\
& \leq (tx+(1-t)y) s_{i} \left( \dfrac{tx^{2}+(1-t)y^{2}}{tx+(1-t)y} \right)
\end{align*}
Because of the well-known inequality $2xy \leq x^{2}+y^{2}$, we have:
\begin{align*}
(tx+(1-t)y)^{2} &= t^{2}x^{2}+(1-t)^{2}y^{2}+2t(1-t)xy \\
& \leq tx^{2}+(1-t)y^{2}\enspace .
\end{align*}
Then, because $s_{i}$ is nonincreasing, we have:
\[
s_{i} \left( \dfrac{tx^{2}+(1-t)y^{2}}{tx+(1-t)y} \right) \leq s_{i}(tx+(1-t)y)
\enspace ,
\] so that:
\[
tf_{i}(x)+(1-t)f_{i}(y) \leq (tx+(1-t)y)s_{i}(tx+(1-t)y) = f_{i}(tx+(1-t)y)
\enspace ,
\]
and $f_{i}$ is concave.
\end{proof} 

\subsection{A tropical representation of customers' response} \label{subsec2-2}

The lower-level component of our bilevel problem can be studied thanks to tropical techniques. Tropical mathematics refers to the study of the max-plus semifield $\Rmax$, that is the set $\Rinfty$ endowed with two laws $\tplus$ and $\tdot$ defined by
\(a \tplus b = \max(a,b)\)
and
\(a \tdot b = a+b\),
see~\cite{bcoq,itenberg2009tropical,butkovicbook,MacLaganSturmfels} for background.
We first consider the relaxation in which the price vector $y $ can take any real value, i.e. $y\in \R^{n}$. Each customer $k$ defines his consumption $u_k^*$ by solving the problem:
\begin{equation}\label{low-level}
\max_{u_{k} \in F_{k}} \sum_{i} \left[ \rho_{k}(i)+ y_{i} \right] u_{k}(i) = \max_{u_{k} \in F_{k}} \langle \rho_{k}+y,u_{k} \rangle \enspace .
\end{equation}
The map $P_k: y \mapsto \max_{u_{k} \in F_{k}} \langle \rho_{k}+y,u_{k} \rangle$ is
convex, piecewise affine, and the gradients of its linear parts are integer valued. It can be thought of as a tropical polynomial function
in the variable $y$. Indeed, with the tropical notation, we have
\begin{align*}
\troppoly{P_k}(y)=\underset{u_k \in F_k}{\bigoplus} \left[ \underset{i \in \cro{n}}{\bigodot} \left( \rho_k(i) \odot y_i \right)^{\odot u_k(i)} \right]
\end{align*}
where $z^{\odot p} := z\odot \dots \odot z = p\times z$ denotes
the $p$th tropical power. In this way, we see
that all the monomials of $P_k$ have degree
$\sum_{i}{u_{k}(i)}=R_{k}$, so that $P_k$ is homogeneous
of degree $R_k$, in the tropical sense.
This remark leads to the following lemma:
\begin{lemma} \label{lemma-relaxed}
Denote by $e=(1 \dots 1) \in \R^{n}$.
Let $y$ be a solution of the relaxation $y \in \R^n$ of Problem~\ref{highlevel3}. 
Then, for all $\beta \in \R$, $y+\beta e$ is a solution of the relaxation $y \in \R^n$ of Problem~\ref{highlevel3}. 
\end{lemma}
\begin{proof}
Consider a solution $y \in \R^{n}$ of the relaxed problem.
Because $P_k$ is homogeneous of degree $R_k$, 
we have for all $\beta \in \R^n$, $P_k(y+\beta e)=P_k(y)+\beta R_k$. In particular:
\[
u_k^* \in \arg\max_{u_k \in F_k} \langle \rho_k+y,u_k \rangle  \Leftrightarrow u_k^* \in \arg\max_{u_k \in F_k} \langle \rho_k+y+\beta e,u_k \rangle 
\]
Hence, $y+\beta e$ leads to the same repartition of the customers $N^{*}$ and corresponds also to an optimal solution of the relaxed bilevel problem. 
\end{proof}

\begin{corollary} \label{lemma-relax}
The bilevel problem~\ref{highlevel3} has the same value as its relaxation $y \in \R^{n}$.
\end{corollary}
\begin{proof} Consider a solution $y^{*} \in \R^{n}$ of the relaxed problem,
and take $\beta \geq - \min_{i} y^{*}_i$.
Then, we have $y^* + \beta e \in R_+^n$ and solution of the relaxed problem according to Lemma~\ref{lemma-relaxed}.
Consequently, $y^* + \beta e$ is a solution of Problem~\ref{highlevel3}.
\end{proof}

By definition, the {\em tropical hypersurface}
associated to a tropical polynomial function is the nondifferentiability locus
of this function. 
Since the monomial $P_k$ is homogeneous, its associated tropical 
hypersurface is invariant by the translation by a constant vector.
Therefore, it can be represented as a subset of the
tropical projective space $\TP^{n-1}$.
The latter is
defined as the quotient of $\R^{n}$ by the equivalence relation
which identifies two vectors which differ by a constant vector, and 
it can be identified to $\R^{n-1}$ by the map 

$ \TP^{n-1}  \to \R^{n-1}$,
$y  \mapsto (y_{i}-y_{n})_{i\in \cro{n-1}}$.

\begin{example} \label{ex-1}
Consider a simple example with $T=3$ time steps (for instance morning, afternoon and evening), $L=1$ (that is $n=3$), $K=5$ and $\mathcal{J}_{k}= \emptyset $ for each $k$. 
The parameters of the customers are 
\begin{align*}
\rho_{1}  &= \left[ 0,0,0 \right],\; R_{1}=1,\quad 
\rho_{2}  =\left[ 0,-1,0 \right],\; R_{2}=2 \enspace ,\\
\rho_{3}  &=\left[ -1,1,0 \right],\; R_{3}=1  \enspace 
\rho_{4}  = \left[ 1/2,1/2,0 \right],\; R_{4}=2,
\enspace \\
\rho_{5}  &= \left[ 1/2,2,0 \right],\; R_{5}=1  \enspace .
\end{align*}
 
The tropical polynomial of the first customer is $\troppoly{P_{1}}(y)=\max \left( y_{1},y_{2},y_{3} \right)$,
meaning that this customer has no preference and consumes when the incentive is the best.
Its associated tropical hypersurface is a tropical line
(since $P_1$ has degree $1$),
so it splits $\TP^{2}$ in three different regions corresponding to a choice of the vector 
$u_{1}$ among $(1,0,0)$, $(0,1,0)$ and $(0,0,1)$, see Figure~\ref{fig-line}.
E.g., the cell labeled by $(1,0,0)$ represents a consumption
concentrated the morning, induced by a price $y_1>y_2$ and $y_1>y_3$.

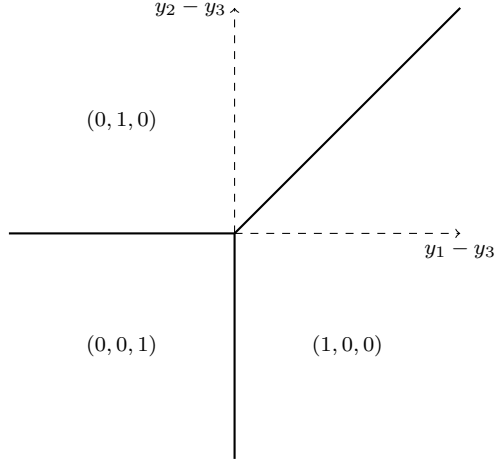
\begin{figure}[h]
\begin{center}
\begin{tikzpicture} [scale=1.5]
    \draw [black,thick] (-2,0)--(0,0);
    \draw [black,thick] (0,-2)--(0,0);
    \draw [black,thick] (0,0)--(2,2) ;
    \draw (-1,-1) node{$(0,0,1)$};
    \draw (-1,1) node{$(0,1,0)$};
    \draw (1,-1) node{$(1,0,0)$};
    \draw [dashed] [->](0,0)--(2,0) node[below]{$y_{1}-y_{3}$};
    \draw [dashed] [->](0,0)--(0,2) node[left]{$y_{2}-y_{3}$};
\end{tikzpicture}
\end{center}
\caption{A customer response: a tropical line splits the projective space into three cells. Each cell corresponds to a possible customer response}
\label{fig-line}
\end{figure}

To study jointly the responses of the five customers, 
we represent the arrangement of the tropical hypersurfaces associated to 
the $P_k, \; k\in \cro{5}$ (see Figure~\ref{fig-five}), with 
\begin{align*}
\troppoly{P_{2}}(y)  &=\max \left( y_{1}+y_{2}-1,y_{1}+y_{3},y_{2}+y_{3}-1 \right) ,\\
\troppoly{P_{3}}(y) & =\max \left( y_{1}-1,y_{2}+1,y_{3} \right) ,\\
\troppoly{P_{4}}(y)  &=\max \left( y_{1}+y_{2}+1,y_{1}+y_{3}+1/2,y_{2}+y_{3}+1/2 \right) ,\\
\troppoly{P_{5}}(y) & =\max \left( y_{1}+1/2,y_{2}+2,y_{3} \right)  \enspace .
\end{align*}


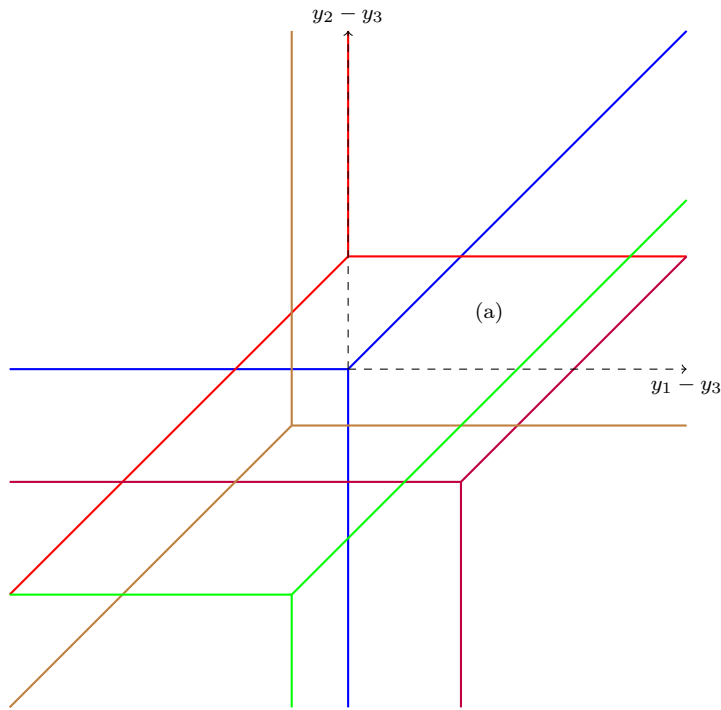
\begin{figure}[h]\vspace{-1em}
\begin{center}
\begin{tikzpicture} [scale=1.5]
   \draw [blue,thick] (0,0)--(-3,0);
    \draw [blue,thick] (0,0)--(0,-3); 
    \draw [blue,thick] (0,0)--(3,3); 
    \draw [red,thick] (0,1)--(0,3); 
    \draw [red,thick] (0,1)--(3,1); 
    \draw [red,thick] (0,1)--(-3,-2); 
    \draw [purple,thick] (1,-1)--(-3,-1); 
    \draw [purple,thick] (1,-1)--(1,-3); 
    \draw [purple,thick] (1,-1)--(3,1); 
    \draw [brown,thick] (-0.5,-0.5)--(-0.5,3); 
    \draw [brown,thick] (-0.5,-0.5)--(3,-0.5); 
    \draw [brown,thick] (-0.5,-0.5)--(-3,-3); 
    \draw [green,thick] (-0.5,-2)--(-3,-2); 
    \draw [green,thick] (-0.5,-2)--(-1/2,-3); 
    \draw [green,thick] (-0.5,-2)--(3,1.5); 
	\draw [black] (1.25,0.5) node{(a)};
    \draw [dashed] [->](0,0)--(3,0) node[below]{$y_{1}-y_{3}$};
    \draw [dashed] [->](0,0)--(0,3) node[above]{$y_{2}-y_{3}$};
\end{tikzpicture}
\end{center}
\vspace{-.6em}
\caption{Arrangement of tropical hypersurfaces: each tropical hypersurface corresponds to a customer response. For example, the cell (a) corresponds to discounts $y$ with responses \color{blue} (1,0,0) for customer 1, \color{red} (1,0,1) for customer 2, \color{purple} (0,1,0) for customer 3, \color{brown} (1,1,0) for customer 4 and \color{green} (0,1,0) for customer 5 \color{black}. Hence, the total number of customers in the network with these discounts is (3,3,1). }\label{fig-five}
\vspace{-.4em}
\end{figure}
\end{example}

\begin{lemma}[Corollary of \protect{\cite[\S 4, Lemma 3.1]{tran2015product}}] \label{tropicallemma}
Each cell of the arrangement of tropical hypersurfaces corresponds to a collection of customers responses $(u_{1},...,u_{K})$ and
 to a unique traffic vector $N$, defined by $N=\sum_{k}{u_{k}}$. 
\end{lemma}

\subsection{Decomposition theorem} \label{subsec2-3}

We next show that the present bilevel problem can be solved by decomposition.
We note that the function to optimize for the higher level problem, i.e. the optimization problem of the provider, depends only on $N$. The variables $y_i$ allow one to generate the different possible vectors $N$.

\begin{definition} \label{def-feasible}
A vector $N \in \Z^n$ is said to be \textit{feasible} if there exists $K$ vectors $u_1^*, \dots, u_K^*$ such that $N=\sum_{k=1}^K u_k^*$ and there exists $y \in \R^n$ such that for each $k \in \cro{K}$,
$u_k^* \in \arg\max_{u_k \in F_k} \langle \rho_k +y,u_k \rangle$.
\end{definition}

So, we will characterize the feasible vectors $N$ in order to optimize directly the satisfaction function on the set of feasible $N$.
We define the relaxation of Problem~\ref{highlevel3} to the case $y \in \R^n$. 

\begin{problem}[Bilevel problem with real discounts] \label{bilevel-telecom}
\begin{equation*}
\max_{y \in \R^n} \sum_{i=1}^n f_i(N_i) \quad 
\text{s.t.} \;  \forall i, N_i \leq N_i^C,
\end{equation*}
with $N=\sum_{k=1}^K u_k^*$ and for all $k \in \cro{K}$, $u_k^*$ solution of:
\begin{equation*} 
\max_{u_k \in F_k} \quad \langle \rho_k + y,u_k \rangle.
\end{equation*}
\end{problem}

According to Lemma~\ref{lemma-relax}, Problem~\ref{highlevel3} has the same value than the relaxation problem~\ref{bilevel-telecom}. 
Moreover, according to Lemma~\ref{lemma-relaxed}, if $(y^*,N^*)$ is an optimal solution of Problem~\ref{bilevel-telecom}, 
then $(y^*+\beta e,N^*)$ is also an optimal solution of Problem~\ref{bilevel-telecom} for every $\beta \in \R$. 
We recall that $e \in \R^n$ is a vector defined by $e^T = (1,\dots,1)$. 
Then, if we find an optimal solution $(y^*,N^*)$ of Problem~\ref{bilevel-telecom}, 
then $(y^*+\beta e,N^*)$ with $\beta = - \min_{i \in \cro{n}} y^*_i$ is a solution of Problem~\ref{bilevel-telecom} such that $y^* + \beta e \in \R_+^n$. 
Consequently, $(y^*+\beta e,N^*)$ is a solution of Problem~\ref{highlevel3}.
Hence, a solution of Problem~\ref{bilevel-telecom} (with real discounts) provides a solution of Problem~\ref{highlevel3} (with nonnegative discounts). 
In the sequel, we will study the bilevel problem \ref{bilevel-telecom}.

Most of the following results are applications of classical notions of convex analysis which can be found in \cite{rockafellar1970convex}.
It is convenient to introduce the convex characteristic function $\chi_A$ of a set $A \subset \R^n$, 
defined by $\chi_A(x)= 0$ if $x \in A$,
and $\chi_A(x)=+\infty$ otherwise. 
If $A$ is a convex set, then $\chi_A$ is a convex function.
We define also for every $k$ the polytope $\Delta_{k}$ as the convex hull of $F_{k}$, together with the convex function $\varphi_k$ defined by
$\varphi_k (u)= - \langle \rho_k,u \rangle + \chi_{\Delta_k}(u)$.

\begin{lemma} \label{lemma-delta}
$\Delta_{k} = \lbrace u \in \left[ 0,1 \right]^{n} | \sum_{i=1}^n u(i)=R_k$ and $\forall i \in \mathcal{J}_{k}, u(i)=0 \rbrace$ and $F_k$ is exactly the set of vertices of $\Delta_k$.
\end{lemma}

\begin{proof}
%

Let us define the polytope $\Delta_k'=\lbrace u \in \left[ 0,1 \right]^n \mid \sum_{i=1}^n u(i)=R_{k} $ and $\forall (t,l) \in \mathcal{J}_{k}, u(i)=0 \rbrace$.
Clearly, $F_k \subset \Delta_k'$.
Then, $\Delta_k \subset \Delta_k'$.


Consider a point $u$ of $\Delta_k'$ which is not in $F_k$. There exists an index $i$ such that $0 < u(i) < 1$. In particular $u(i) \notin \N$. However, $\sum_{i} u(i)  = R_{k} \in \N$. So, there exists another index $j$ such that  $0< u(j) < 1$. Hence, there exists $\varepsilon > 0$ such that the points $u^{-}$ ans $u^{+}$ defined by:
\[
\left\lbrace \begin{array}{c}
u^{-}(i)=u(i)-\varepsilon \text{ and } u^{+}(i)=u(i)+\varepsilon \\
u^{-}(j)=u(j)+\varepsilon \text{ and } u^{+}(j)=u(j)-\varepsilon \\
u^{-}(k)=u^{+}(k)=u(k) \text{ otherwise }
\end{array} \right.
\] are in $\Delta_k'$. Because $x=\frac{x^{-}+x^{+}}{2}$ with $x \neq x^{-}$ and $x \neq x^{+}$, $x$ is not a vertex of $\Delta_k'$.
Consequently, the set of vertices of $\Delta_k'$ is included in $F_k$. 
Because $\Delta_k'$ is the convex hull of its vertices, we have $\Delta_k' \subset \Delta_k$. 

The polytope $\Delta_k$ is such that $\Delta_k=\{ u \in \R^n \mid 0 \leq u \leq a \; \text{and} \; e^Tu=R_k \}$,
with $a(i)=0$ if $i \in \mathcal{J}_k$ and $a(i)=1$ otherwise, and
$e^T=(1, \dots ,1) \in \mathcal{M}_{1,n}(\R)$. 
Then, because $e^T$ is a totally unimodular matrix, the vertices of $\Delta_k$ are exactly its integer points, 
that is $F_k$.
\end{proof}

\begin{corollary} \label{cor-fenchel}
The value of each low level problem~\ref{lowlevel3} is the value of the Legendre-Fenchel transform of $\varphi_{k}$ at point $y$, i.e. 
$ \varphi_{k}^{*}(y) = \sup_{u_{k} \in \Delta_{k}} \left[ \langle y,u_{k} \rangle - \varphi_{k}(u_{k}) \right] $.
\end{corollary}

\begin{proof} 
The vertices of $\Delta_{k}$ are $F_{k}$. Hence:
\[
\max_{u_k \in F_k} \langle \rho_k+ y,u_k \rangle 
= \sup_{u_k \in \Delta_k} \langle \rho_k+ y,u_k \rangle
= \sup_{u_k \in \Delta_k} \langle y,u_k \rangle - \varphi_k(u_k)
\]
\end{proof}

We want to characterize the feasible vectors.
We have first the following result.

\begin{lemma} \label{lemma-realfeasible}
Let $N$ be a real vector.
Then, there exists $y \in \R^n$ and $u_1^*, \dots, u_K^*$ such that $N=\sum_{k \in \cro{K}} u_k^*$ and for every $k \in \cro{K}$, $u_k^* \in \arg\max_{u_k \in \Delta_k} \langle \rho_k+y,u_k \rangle$  if and only if $N \in \sum_{k \in \cro{K}} \Delta_k$.
\end{lemma}

\begin{proof}
Such vectors $u_k^*$ belong to $\Delta_k$, so $N \in \sum_{k \in \cro{K}} \Delta_k$. 

Let $k \in \cro{K}$ and $y \in \R^n$.
A vector $u_k^* \in \Delta_k$ is such that $u_k^* \in \arg\max_{u_k \in \Delta_k} \langle \rho_k+y,u_k \rangle$ if and only if $u_k^* \in \partial \varphi_{k}^{*}(y)$,
where $\partial \varphi_{k}^{*}$ denotes the subdifferential of the convex function $\varphi_{k}^{*}$.
Then, a vector $N = \sum_k u_k^*$ if and only if  $N \in \sum_k \partial \varphi_k^*(y)$.
By~\cite[Th.\ 23.8]{rockafellar1970convex}, 
$\sum_k \partial \varphi_k^*(y) =  \partial \left( \sum_{k}{ \varphi_{k}^{*}}\right)(y) = \partial \psi^{*}(y) $,
where $\psi= \underset{k}{\square} \varphi_{k}$ is the inf-convolution of the functions $\varphi_{k}$. 

Let $N$ be a real vector. 
Then, there exists $y \in \R^n$ and $u_1^*, \dots, u_K^*$ such that $N=\sum_{k \in \cro{K}} u_k^*$ and for every $k \in \cro{K}$, $u_k^* \in \arg\max_{u_k \in \Delta_k} \langle \rho_k+y,u_k \rangle$  if and only if $N \in \partial \psi^*(y)$,
or equivalenty $y \in \partial \psi(N)$ (because $\psi$ is convex), 
that is if and only if $\partial \psi(N) \neq \emptyset$.
The function $\psi$ is polyhedral (as the inf-convolution of polyhedral convex functions) and it is finite at every point in $\sum_{k}{\Delta_{k}}$. 
So, $\forall N' \in\sum_{k}{\Delta_{k}}, \partial \psi(N')$ is a non-empty polyhedral convex set \cite[Th.\ 23.10]{rockafellar1970convex}.
The result comes straightforwardly.
\end{proof}

It is now possible to characterize the feasible vectors.
\begin{lemma} \label{lemma-feasible}
A vector $N \in \Z^n$ is feasible if and only if $N \in \sum_k F_k$.
\end{lemma}

\begin{proof}
According to Definition~\ref{def-feasible}, a vector $N \in \Z^n$ is feasible if and only if there exists $y \in \R^n$ and $K$ vectors $(u_k^*)_{k \in \cro{K}}$ such that $N=\sum_k u_k^*$ and $u_k^* \in \arg\max_{u_k \in F_k} \langle \rho_k+y,u_k \rangle$. 
As a consequence of Lemma~\ref{lemma-delta}, $\arg\max_{u_k \in F_k} \langle \rho_k+y,u_k \rangle = \arg\max_{u_k \in \Delta_k} \langle \rho_k+y,u_k \rangle$.
Then, by Lemma~\ref{lemma-realfeasible}, a vector $N \in \Z^n$ is feasible if and only if $N \in (\sum_{k \in \cro{K}} \Delta_k) \cap \Z^n$.
We have now to prove $\sum_{k \in \cro{K}} F_k = (\sum_{k \in \cro{K}} \Delta_k) \cap \Z^n$.
Because $F_k = \Delta_k \cap \Z^n$, the inclusion  $\sum_{k \in \cro{K}} F_k  \subset (\sum_{k \in \cro{K}} \Delta_k) \cap \Z^n$ is obvious.
Conversely, consider $N \in (\sum_{k \in \cro{K}} \Delta_k) \cap \Z^n$. Then, the set $\Delta_N = \{ (u_1, \dots, u_K) \in \Delta_1 \times \dots \times \Delta_K \mid \sum_{k=1}^K u_k = N \}$ is a non-empty polytope. 
A vector $u = (u_1, \dots, u_K)$ belongs to $\Delta_N$ if it satisfies the following constraints:
 \begin{align*}
\left\lbrace \begin{array}{l}
\forall k,i,\;  0 \leq u_{k}(i) \leq \mathds{1}_{i \in \mathcal{J}_k} \enspace ,\\
\forall k,\; \sum_{i}{u_{k}(i)}=R_{k} \enspace ,\\
\forall i,\; \sum_{k}{u_{k}(i)}=N_{i} \enspace .
\end{array} \right.
\end{align*},
that is $\Delta_N = \{ u \in \R^{Kn} \mid 0 \leq u \leq a, Au = b \}$, with $a \in \R^{Kn}$ such that $a_k(i) =\mathds{1}_{i \in \mathcal{J}_k}$ for every $i \in \cro{n}, k \in \cro{K}$, and $A \in \mathcal{M}_{K+n,Kn}(\Z)$ and $b \in \Z^{K+n}$ defined by:
\[ 
A = \left( \begin{array}{ccccccccccccc}
1 & 1 & ... & 1 & 0 & 0 & ... & 0 & ... & 0 & 0 & ... & 0  \\
0 & 0 & ... & 0 & 1 & 1 & ... & 1 & ... & 0 & 0 & ... & 0  \\
\multicolumn{13}{c}{...} \\
0 & 0 & ... & 0 & 0 & 0 & ... & 0 & ... & 1 & 1 & ... & 1  \\
-1 & 0 & ... & 0 & -1 & 0 & ... & 0 & ... & -1 & 0 & ... & 0  \\
0 & -1 & ... & 0 & 0 & -1 & ... & 0 & ... & 0 & -1 & ... & 0  \\
\multicolumn{13}{c}{...} \\
0 & 0 & ... & -1 & 0 & 0 & ... & -1 & ... & 0 & 0 & ... & -1  \\
\end{array} \right) \text{ and } b= \left( \begin{array}{c}
R_{1} \\
R_{2} \\
... \\
R_{K} \\
-N_{1} \\
-N_{2} \\
... \\
-N_{n} \\
\end{array} \right)
\]
By Poincaré's lemma, $A$ is totally unimodular. 
In particular, the extreme points of $\Delta_N$ are integer. 
Then, there exists $(u_1^*, \dots, u_K^*)$ with for every $k \in \cro{K}$, $u_k^* \in \Delta_k \cap \Z^n = F_k$ such that $N=\sum_k u_k^*$.
\end{proof}

Each vector $N \in \sum_k F_k$ can be written as sum of vectors $u_k^* \in F_k$ for $k \in \cro{K}$ such that there exists $y \in \R^n$ with $u_k^* \in \arg\max_{u_k \in F_k} \langle \rho_k+y,u_k \rangle$.
In order to determine such vectors $u_k^*$, we have the following lemma:

\begin{lemma}\label{lemma-infconv}
Let $N=\sum_{k} u_k^*$ with $u_{k}^* \in \Delta_{k}\; \forall k$. The following assertions are equivalent:
\begin{enumerate}
\item There exists $y \in \R^n$ such that for each $k \in \cro{K}$, $u_k^* \in \arg\max_{u_k \in \Delta_k} \langle \rho_k+y,u_k \rangle$.
\item The vectors $u_{1}^*,\dots,u_{K}^*$ realize the minimum in the inf-convolution $\psi$, i.e. $$\psi(N)=-\sum_{k}{\langle \rho_{k},u_{k}^* \rangle}$$.
\end{enumerate}
\end{lemma}
\begin{proof}
(1) $\Rightarrow$ (2) : We have for every $k$:
\[
\langle \rho_{k}+y,u_{k}^* \rangle = \sup_{u_{k} \in \Delta_{k}} \left[ \langle \rho_{k}+y,u_{k} \rangle  \right] 
\]
By summing those equalities, we have:
\begin{align*}
\langle y,N \rangle + \sum_{k} \langle \rho_{k},u_{k}^* \rangle   & = \sum_{k} \sup_{u_{k} \in \Delta_{k}} \left[ \langle \rho_{k}+y,u_{k} \rangle  \right]  \\
& = \sup_{u_{1} \in \Delta_{1},\dots,u_{K} \in \Delta_{K}}\sum_{k}\left[ \langle \rho_{k}+y,u_{k} \rangle  \right] 
\end{align*}
By considering only the vectors $u_{1} \in \Delta_{1},\dots,u_{K} \in \Delta_{K}$ such that $\sum_{k} u_{k}=N$, we can write $\sum_{k} \langle \rho_{k},u_{k}^* \rangle= \underset{\substack{
u_{1} \in \Delta_{1},\dots,u_{K} \in \Delta_{K} \\
\sum_{k}{u_{k}}=N 
}}{\sup}\sum_{k} \langle \rho_{k},u_{k} \rangle$ which is exactly the second assertion.

(2) $\Rightarrow$ (1): The set $\partial \psi (N)$ is non-empty. Consider $y \in \partial \psi (N)$, that is $N \in \partial \psi^{*} (y)$. We can write:
\[
\forall N' \in \sum_{k}{\Delta_{k}}, \; \langle y,N \rangle - \psi (N) \geq \langle y,N' \rangle - \psi (N')
\]
So:
\begin{align*}
& \sum_{k}{\langle \rho_{k}+y,u_{k}^* \rangle}  = \langle y,N \rangle + \sum_{k}{\langle \rho_{k},u_{k}^* \rangle} = \langle y,N \rangle - \psi (N) \\ 
& = \sup_{N' \in \sum_{k}{\Delta_{k}}}{\left[ \langle y,N' \rangle - \psi (N')  \right] } \\
 & = \sup_{N' \in \sum_{k}{\Delta_{k}}}{\left[ \langle y,N' \rangle + \underset{\substack{
u_{1} \in \Delta_{1},\dots,u_{K} \in \Delta_{K} \\
\sum_{k}{u_{k}}=N' 
}}{\sup}\sum_{k}{ \langle \rho_{k},u_{k} \rangle}  \right] } \\
& = \sup_{u_{1} \in \Delta_{1},\dots,u_{K} \in \Delta_{K}}\sum_{k} \langle \rho_{k}+y,u_{k} \rangle   \\
& = \sum_{k} \underset{u_{k} \in \Delta_{k}}{\sup} \langle \rho_{k}+y,u_{k} \rangle  \\
\end{align*}

Consequently, if one $u_{k}^*$ is not an optimal solution of the low-level problem, the previous equality cannot be true. 
\end{proof}

The high-level problem of Problem~\ref{bilevel-telecom} consists in maximizing a function depending only on a vector $N$ which has to be a feasible vector. 
It is now possible to write the main theorem of this section, 
which establishes a decomposition method for solving Problem~\ref{bilevel-telecom}.

\begin{theorem}{\rm (Decomposition)} The bilevel problem~\ref{bilevel-telecom} can be solved as follows:
\begin{enumerate}
\item Find an optimal solution $N^{*}$ to the high level problem with unknown $N$:
\begin{align} \label{high-level}
\max_{N \in \sum_{k}{F_{k}}} & \sum_{i=1}^n {f_i(N_i)}  \; \;
\text{ s.t. } \forall i, \; N_{i} \leq N_{i}^{C} \enspace .
\end{align}
\item Find vectors $(u_1^*,\dots,u_K^*)$ solutions of the following problem:
\[
\max_{
\scriptstyle u_{1} \in F_{1},...,u_{K} \in F_{K} \atop
\scriptstyle
\sum_{k}{u_{k}}=N^{*} 
}
\sum_{k}{\langle \rho_{k},u_{k} \rangle}  \enspace .
\]
\item Find a vector $y^{*}$ such that $\forall k$, $u_{k}^{*}$ is a solution of the low level problem. 
\end{enumerate} \label{decomposition}
\end{theorem}

\begin{proof}
The bilevel programming problem~\ref{bilevel-telecom} can be rewritten $\max_{N \text{ feasible}} \sum_i f_i(N_i)$ subject to $\forall i \in \cro{n}, \; N_i \leq N_i^C$. 
According to Lemma~\ref{lemma-feasible}, $N$ is feasible if and only if $N \in \sum_{k}{F_{k}}$.
So, a necessary condition for a vector $y^{*}$ to be an optimal solution of the bilevel problem is that for every $k$, there exists $u_{k}^{*} \in \partial \varphi_{k}^{*}(y^{*})$ such that $N^{*}=\sum_{k}{u_{k}^{*}}$ is an optimal solution of the problem:
\begin{align*}
\max_{N \in \sum_{k} F_{k}} & \sum_{i}f_{i}(N_{i}) \\
\text{ s.t. } & \forall i , \; N_{i} \leq N_{i}^{C}
\end{align*}

After finding $N^{*}$, it is possible to find $u_{k}^{*} \in \partial \varphi_{k}^{*}(y^{*})$ by solving the inf-convolution problem as a consequence of Lemma~\ref{lemma-infconv}.
Because $u_{k}^{*} \in \partial \varphi_{k}^{*}(y^{*})$ is equivalent to $y^{*} \in \partial \varphi_{k}(u_{k}^{*})$, each point of $\bigcap_{k} \; \partial \varphi_{k}(u_{k}^{*})$ is an optimal solution of the bilevel problem. 

\end{proof}

The second step of this theorem consists in solving a linear program. We next show that the third step reduces to a linear feasibility problem. 

\begin{lemma} \label{linear-feasibility}
Let $N \in \Z^n$ be a feasible vector and $u_k^* \in F_k$ ($k \in \cro{K}$) be vectors such that $N=\sum_k u_k^*$ and $\psi(N)=-\sum_k \langle \rho_k,u_k^* \rangle$. 
Then, the set of vectors $y^* \in \R^n$ such that $\forall k \in \cro{K}$, $u_k^* \in \arg\max_{u_k \in F_k} \langle \rho_k+y^*,u_k \rangle$ is non-empty and is the polytope defined by the following inequalities:
\[
\forall k \in \cro{K}, \; \forall i,j \notin \mathcal{J}_k, \text{ such that } u_k^*(i)=1,u_k^*(j)=0, \;  \rho_k(i)+y^*_i \geq \rho_k(j)+y^*_j
\]
\end{lemma}

\begin{proof} 
According to Lemma~\ref{lemma-infconv}, there exists $y^{*} \in \bigcap_{k} \partial \varphi_{k}(u_{k}^{*})$. 
Hence, we have $\forall u_{k} \in F_{k}$, $\langle \rho_{k}+y^{*},u_{k}^{*} \rangle \geq \langle \rho_{k}+y^{*},u_{k} \rangle$. 

Consider indices $i,j \notin \mathcal{J}_{k}$ with $u_{k}^{*}(i)=1$, $u_{k}^{*}(j)=0$,
and the vector $u_{k}$ defined by $u_{k}(i)=0$, $u_{k}(j)=1$ and $\forall l \neq i,j, \; u_{k}(l)=u_{k}^{*}(l)$. 
We verify easily $u_{k} \in F_{k}$, so that the condition $\langle \rho_{k}+y^{*},u_{k}^{*} \rangle \geq \langle \rho_{k}+y^{*},u_{k} \rangle$,
which can be rewritten $\rho_{k}(i)+y^{*}_{i} \geq \rho_{k}(j)+y^{*}_{j}$,
is satisfied.

Moreover, this condition is sufficient. Consider $y^{*}$ such that $\forall i,j \notin \mathcal{J}_{k}$ with $u_{k}^{*}(i)=1$, $u_{k}^{*}(j)=0$, we have $\rho_{k}(i)+y^{*}(i) \geq \rho_{k}(j)+y^{*}(j)$.
Consider $u_{k} \in F_{k}$. 
By definition of $F_{k}$, the quantity $\langle \rho_{k}+y^{*},u_{k} \rangle$ corresponds to the sum of $R_{k}$ coordinates of $\rho_{k}+y^{*}$ for which the index is not in $\mathcal{J}_{k}$.
Hence,
\begin{align*}
\langle \rho_{k}+y^{*},u_{k} \rangle & = \sum_{i, u_{k}(i)=1, u_{k}^{*}(i)=1} \left( \rho_{k}(i)+y^{*}_{i} \right) + \sum_{j, u_{k}(j)=1, u_{k}^{*}(j)=0} \left(\rho_{k}(j)+y^{*}_{j} \right) \\
& \leq \sum_{i, u_{k}(i)=1, u_{k}^{*}(i)=1} \left( \rho_{k}(i)+y^{*}_{i} \right) + \sum_{j, u_{k}(j)=0, u_{k}^{*}(j)=1} \left(\rho_{k}(j)+y^{*}_{j} \right) = \langle \rho_{k}+y^{*},u_{k} \rangle
\end{align*}
because of the lemma hypothesis and 
because $\# \{j|u_{k}(j)=1, u_{k}^{*}(j)=0 \} = \# \{j|u_{k}(j)=0, u_{k}^{*}(j)=1 \}$

\end{proof}

For every $k$, the latter inequalities define a polytope, and we have to find $y^{*}$ in the intersection of all these polytopes. 

\section{A first algorithm} \label{sec-first}

In this section, we explain how the decomposition method provided by Theorem~\ref{decomposition} leads to a polynomial time algorithm for solving Problem~\ref{bilevel-telecom}.
We will use some elements of discrete convexity developed by Danilov, Koshevoy \cite{danilov2004discrete} and Murota \cite{murota2003discrete}, that we recall first. 
We next explain how to solve Problem~\ref{bilevel-telecom}.

An integer set $B \subset \Z^{n}$ is
{\em $M$-convex}~\cite[Ch.\ 4, p.101]{murota2003discrete} if
$
\forall x,y \in B, \forall i\in [n]$ such that
$x_{i}>y_{i}, \exists j\in [n]$ such that
$x_{j}<y_{j}$,
$x-e_{i}+e_{j} \in B$
and
$y+e_{i}-e_{j} \in B$,
where $e_{i}$ is the $i$-th vector of the canonical basis in $\R^{n}$. 
\begin{lemma} \label{feasibledomain}
The feasible domain of the high-level program $$B=\{ N \in \sum_{k}{F_{k}} | \forall i\; N_{i} \leq N_{i}^{C} \}$$ is a $M$-convex set
of $\Z^{n}$.
\end{lemma}

\begin{proof}
We can check easily that $\forall k$, the set $F_{k}$ is $M$-convex.
Taking two different vectors $u_{k}$ and $v_{k}$ in $F_{k}$, there exist $i,j$ such that $u_{k}(i)=1, v_{k}(i)=0$ and $u_{k}(j)=0, v_{k}(j)=1$. These indices $i,j$ do not belong to $\mathcal{J}_{k}$. The vectors $u_{k}-e_{i}+e_{j}$ and $v_{k}+e_{i}-e_{j}$ have coordinates in $\{ 0,1 \}$ with a sum equal to $R_{k}$ and all coordinates in $\mathcal{J}_{k}$ equal to 0. 

It is known that a Minkowski sum of $M$-convex sets is $M$-convex
\cite[Th.\ 4.23, p.115]{murota2003discrete}, and so
the set $\sum_{k}{F_{k}}$ is $M$-convex.

Finally, consider two vectors $N$ and $N'$ of $B$. They belong to $\sum_{k}{F_{k}}$, so for each $i$ with $N_{i} > N'_{i}$, we can find $j$ with $N_{j} < N'_{j}$ such that $N-e_{i}+e_{j}$ and $N'+e_{i}-e_{j}$ are in $\sum_{k}{F_{k}}$. The $i$-th coordinate of $N-e_{i}+e_{j}$ is $N_{i}-1 < N_{i} \leq N_{l}^{C}$ and the $j$-th coordinate of $N-e_{i}+e_{j}$ is $N_{j}+1 \leq N'_{j} \leq N_{j}^{C}$. So $N-e_{i}+e_{j} \in B$ and similarly $N'+e_{i}-e_{j} \in B$, which proves the $M$-convexity of $B$.

\end{proof} 

A function $g: \Z^n \mapsto \R$ is $M$-convex \cite[ch.\ 6.1, p.133]{murota2003discrete} if $\forall x,y \in \Z^n$ such that $g(x)$ and $g(y)$ are finite real values, $\forall i \in \cro{n}$ such that $x_i > y_i$, 
$\exists j \in \cro{n}$ such that $x_j < y_j$ and the following condition holds true:
\[
g(x)+g(y) \geq g(x-e_i+e_j)+g(y+e_i-e_j)
\]
A function $g$ is $M$-concave if $-g$ is $M$-convex.
It follows from this definition that if $B$ is a $M$-convex set, then $\chi_B$ is a $M$-convex function 
(we recall that $\chi_B: \Z^n \mapsto \R$ is defined by $\chi_B(x)=0$ if $x \in B$ and $\chi_B(x)=+ \infty$ otherwise).
An important property of $M$-convex functions is that local optimality guarantees global optimality \cite[Th.\ 6.26, p.148]{murota2003discrete} in the following sense. 
Let $g$ be a $M$-convex function and $x \in \Z^n$. 
Then $g(x)=\min_{y \in \Z^n} g(y)$ if and only if $\forall i,j \in \cro{n}, \; g(x) \leq g(x-e_i+e_j)$. 

According to Theorem~\ref{decomposition}, we have to solve $\max_{N \in \Z^n} f(N) - \chi_B(N)$, where $f: N \mapsto \sum_i f_i(N_i)$ is a separable concave function, and $B$ is the $M$-convex set introduced in Lemma~\ref{feasibledomain}.
The function $f-\chi_B$ is $M$-concave \cite[Th.\ 6.13.(4), p.143]{murota2003discrete}. 
Then, we have the following result as a direct consequence of \cite[Th.\ 6.26, p.148]{murota2003discrete} :
\begin{theorem} 
Let $N^{*}\in B$. 
Then, $N^*$ is a maximum point of $f$ over $B$
if and only if
$\forall i,j\in \cro{n}$ such that $N^{*}-e_{i}+e_{j} \in B, f(N^{*}-e_{i}+e_{j}) \leq f(N^{*})$.
\end{theorem}


Moreover, Murota (\cite{murota2003discrete}, ch.10, p.281) gives an algorithm which runs in pseudo-polynomial time to minimize $M$-convex functions (see Algorithm~\ref{algo1}).
\begin{algorithm}
\begin{enumerate}
\item Find $x \in \Z^n$ such that $g(x) < +\infty$;
\item Find $i,j\in \arg\min_{k,l \in \cro{n}} g(x-e_k+e_l)$;
\item If $g(x-e_{i}+e_{j}) \geq g(x)$ then stop ($x$ is a global minimizer of $g$);
\item Else $x:=x-e_{i}+e_{j}$ and go back to Step 2;
\end{enumerate}
\caption{Murota's greedy algorithm to minimize a $M$-convex function $g$.} \label{algo1}
\end{algorithm}

By adding a priority rule in Step 2 of Algorithm~\ref{algo1} in the case where $\argmin_{k,l \in \cro{n}} f(x-e_k+e_l)$ is not reduced to a single point, 
a global minimizer of $f$ is obtained by Algorithm~\ref{algo1}in pseudo-polynomial time. 

\begin{proposition}[\cite{murota2003discrete}, Prop.10.2] \label{comp-algogreedy}
Assume that $\dom f$ is bounded. Let $F$ be the number of arithmetic operations needed to evaluate $f$ and $K_1 = \max( ||x-y||_1 \mid x,y \in \dom f)$. 
Then, if a vector in $\dom f$ is given, Algorithm~\ref{algo1} finds a global minimizer of $f$ in $O(Fn^2K_1)$ time. 
\end{proposition}

However, the minimization of a $M$-convex function can be achieved in polynomial time. 

\begin{proposition}[\cite{murota2003discrete}, Prop.10.4] \label{comp-algomurota}
Assume that $\dom f$ is bounded. Let $F$ be the number of arithmetic operations needed to evaluate $f$ and $K_\infty = \max( ||x-y||_\infty \mid x,y \in \dom f)$. 
Then, if a vector in $\dom f$ is given, a global minimizer of $f$ can be found in $O(Fn^3 \log^2(K_\infty/n))$ time. 
\end{proposition}

The different algorithms developed by Murota \cite[Section 10.1]{murota2003discrete} provide a minimizer of a $M$-convex function in polynomial time, 
if an initial point is given and if the domain of the function is bounded. 
Whereas it is trivial to find a vector of $\Z^n$ such that $\forall_i, N_i \leq N_i^C$ or a vector $N$ belonging to $\sum_k F_k$, 
it is not obvious to find one satisfying both conditions.
In fact, such a point can be obtained by solving the minimization problem:
\[
\min_{N \in \sum_k F_k} \sum_i \max(N_i- N_i^C,0)
\]  
The condition $N \in B$ is equivalent to $N \in \argmin_{N \in \sum_k F_k} \sum_i \max(N_i- N_i^C,0)$ if $B$ is non-empty. 
The function $N \mapsto  \sum_i \max(N_i- N_i^C,0)$ is separable convex. 
Then, the function $N \mapsto  \sum_i \max(N_i- N_i^C,0) + \chi_{\sum_k F_k}$ is $M$-convex according to \cite[Th.\ 6.13.(4), p.148]{murota2003discrete}.
Because $\sum_k F_k$ is bounded and a point in $\sum_k F_k$ can be obtained in $O(Kn)$ operations by summing vectors taken in each set $F_k$, 
it is possible to find a point $N^0 \in \argmin_{N \in \sum_k F_k} \sum_i \max(N_i- N_i^C,0) = B$ in polynomial time, by Proposition~\ref{comp-algomurota}.

We can finally write the following result about the complexity of the decomposition method given by Theorem~\ref{decomposition}.

\begin{theorem} \label{th-compdecgeneral}
Let $R = \sum_k R_k$, for every $k \in \cro{K}$, $n_k = n - \# \mathcal{J}_k$ and $\overline{R}=\sum_k R_k(n_k-R_k)$.
An optimal solution of Problem~\ref{bilevel-telecom} can be obtained in $O((Kn)^{3.5} L n^3 \log^2 (K/n) + (n+\overline{R})^{3.5}L)$ arithmetic operations,
where $L$ is the input size of the bilevel problem. 
\end{theorem}

\begin{proof}
The first step of Theorem~\ref{decomposition} is a maximization of a $M$-concave function over a bounded domain $B$.
Finding a point in $B$ can be done by solving the $M$-convex minimization problem:
\[
\min_{N \in \sum_k F_k} \sum_i \max(N_i- N_i^C,0)
\]
The domain of the function $N \mapsto \sum_i \max(N_i- N_i^C,0) + \chi_{\sum_k F_k}$ is $\sum_k F_k$. 
We define $K^1_\infty$ by:
\[
K^1_\infty = \max \{||N-N'||_\infty \mid \; N,N' \in \sum_k F_k \} 
\]
For every $N \in \sum_k N_k$, the entries of $N$ are sum of $K$ binary values. 
Then, $K^1_\infty \leq K$. 
We have to estimate the number of operations $F^1$ needed to evaluate the function $N \mapsto \sum_i \max(N_i- N_i^C,0) + \chi_{\sum_k F_k}$. 
The function $N \mapsto \sum_i \max(N_i- N_i^C,0)$ can be evaluated in $O(n)$ operations. 
As a consequence of Lemma~\ref{lemma-feasible} and Lemma~\ref{lemma-infconv}, $\sum_k F_k = (\sum_k \Delta_k) \cap \Z^n$. 
Hence, for any vector $N$, the conditions $N \in \sum_k F_k$ is equivalent to $N \in (\sum_k \Delta_k) \cap \Z^n$. 
A vector $N$ belongs to $\sum_k \Delta_k$ if there exists for every $k \in \cro{K}$ a vector $u_k \in \Delta_k$ such that $\sum_k u_k = N$. 
Hence, to know whether $N$ belongs to $\sum_k \Delta_k$ or not is a linear feasibility problem in dimension $Kn$, 
It can be solved in $O((Kn)^{3.5} L)$ arithmetic operations by an interior point method (\cite{renegar1988polynomial}).
Here $L$ is the input size of the linear program.
Consequently, $F^1 = O((Kn)^{3.5} L)$, 
and a point in $B$ can be obtained in $O((Kn)^{3.5} L n^3 \log^2 (K/n))$ by Theorem~\ref{comp-algomurota}. 

After obtaining a point in $B$, the first step of Theorem~\ref{decomposition} consists in solving the $M$-concave maximization problem:
\[
\max_{N \in B} \sum_i f_i(N_i).
\]
The domain of the function $N \mapsto \sum_i f_i(N_i) -\chi_B(N)$ is bounded and equal to $B$. 
We define $K^2_\infty$ by:
\[
K^2_\infty = \max \{||N-N'||_\infty \mid \; N,N' \in B \} 
\]
For every $N \in \sum_k N_k$, the entries of $N$ are sum of $K$ binary values. 
Then, for every $i \in \cro{n}$, we have $N_i \leq \min(K,N_i^C)$
Then, $K^2_\infty \leq \min(K,\overline{N}^C)$, 
with $\overline{N}^C = \max_{i \in \cro{n}} N_i^C$. 
The number of operations $F^2$ needed to evaluate the function $N \mapsto \sum_i \max(N_i- N_i^C,0) - \chi_{B}(N)$ is $O((Kn)^{3.5} L)$ like previously. 
Hence, a point $N^* \in \argmax_{N \in B} \sum_i f_i(N_i)$ can be obtained in $O((Kn)^{3.5} L n^3 \log^2 (\min(K,\overline{N}^C)/n))$ by Theorem~\ref{comp-algomurota}.

According to the proof of Lemma~\ref{lemma-feasible}, the second step of Theorem~\ref{decomposition} is a linear program in dimension $Kn$. 
In fact, we have:
\[
\max_{\substack{u_{1} \in F_{1},...,u_{K} \in F_{K} \\\sum_{k}{u_{k}}=N^{*}}}
\sum_{k}{\langle \rho_{k},u_{k} \rangle} = 
\max_{\substack{u_{1} \in \Delta_{1},...,u_{K} \in \Delta_{K} \\\sum_{k}{u_{k}}=N^{*}}}
\sum_{k}{\langle \rho_{k},u_{k} \rangle}
 \enspace ,
\]
and the extreme points of the polyhedron $\Delta_N$ defined by:
\[
\Delta_N = \{ (u_1,\dots,u_K) \in \Delta_1 \times \Delta_K \mid \sum_k u_k = N \}
\]
are integer.
Hence, the second step of Theorem~\ref{decomposition} can be solved in $O((Kn)^{3.5}L)$ arithmetic operations. 

The third step of Theorem~\ref{decomposition} is a linear program in $n$ variables.
For some $u_k^* \in F_k$, the constraints of this program are:
\[
\forall k \in \cro{K}, \; \forall i,j \notin \mathcal{J}_k, \text{ such that } u_k^*(i)=1,u_k^*(j)=0, \;  \rho_k(i)+y^*_i \geq \rho_k(j)+y^*_j.
\]
For every $k \in \cro{K}$, the number of entries of $u_k^*$ equal to $1$ is $R_k$, 
and the number of entries of $u_k^*$ equal to $0$ and which do not belong to $\mathcal{J}_k$ is $n_k$. 
Hence, the number of inequality constraints of this linear program is $\sum_k R_k(n_k-R_k) = \overline{R}$. 
Hence, a solution of this linear program can be found in $O((n+\overline{R})^{3.5} L)$ by interior-point methods. 
\end{proof}

\section{A faster algorithm for solving the bilevel problem} \label{subsec2-4}

\subsection{A polynomial time algorithm for the bilevel problem}

Algorithm~\ref{algo1} can be applied to solve problem~\eqref{high-level} of Theorem~\ref{decomposition}, 
that is maximizing the $M$-concave function $f-\chi_B$,
or equivalently minimizing the $M$-convex function $-f+\chi_B$.

Step 1 consists in finding an initial vector $N \in B$. 
As explained in Section~\ref{sec-first}, this can be done by solving a $M$-convex minimization problem. 
Another approach consists in replacing the function $f-\chi_B$ by $g: N \mapsto f(N)-\chi_{\sum_k F_k}(N) - M \sum_i \max(N_i- N_i^C,0)$,
where $M > 0$ is an integer. 
If $N \in B$, then $g(N)=f(N)$.
If $M$ is sufficiently large, then $M \sum_i \max(N_i- N_i^C,0) \geq M$ if $N \notin B$, and the maximum of the function $g$ is attained for $N \in B$. 
Moreover $N \mapsto M \sum_i \max(N_i- N_i^C,0)$ is separable convex, then $g$ is $M$-concave according to \cite[Th.\ 6.13.(4), p.148]{murota2003discrete}.
Then, both problems $\max_{N \in B} f(N)$ and $\max_{N \in \Z^n} g(N)$ are equivalent, and we can
apply Algorithm~\ref{algo1} to solve the problem $\max_{N \in \Z^n} g(N)$. 
An initial point is obtained by taking any point in $\sum_k F_k$.

We need first part is to determine the number $F$ of operations to evaluate $g$.
Because the different functions $f_i$ are known,
we have to determine the number of operations
to decide whether a vector $N$ belongs to $\sum_k F_k$ or not. 
More precisely, the different evaluations of $f-\chi_B$ are done in Step 2. 
Hence, the question is the following: given a vector $N \in \sum_{k}{F_{k}}$, how many operations are needed to check  whether $N-e_{i}+e_{j}$ (for $i,j \in \cro{n}$) belongs to $\sum_{k}{F_{k}}$.
We next show that this problem can be studied as a shortest path problem in a graph.
Consider $N \in \sum_k F_k$ and let us define $u_k^* \in F_k$ for $k \in \cro{K}$ such that $\psi(N)= \sum_k \langle \rho_k,u_k^* \rangle$,
that is an optimal decomposition of $N$ in Theorem~\ref{decomposition}. 
For each $k \in \cro{K}$ and $\alpha, \beta \in \cro{n}$, we define by $w^k_{\alpha \beta}$ the following quantity:
$w^k_{\alpha\beta}=\rho_k(\alpha)-\rho_k(\beta)$ if $u_k^*(\alpha)=1$ and $u_k^*(\beta)=0$,
and $w^k_{\alpha\beta}=+\infty$ otherwise. 
Then, we define for each $\alpha, \beta \in \cro{n}$, $w_{\alpha\beta}=\min_{k \in \cro{K}} w^k_{\alpha\beta}$.
We consider the oriented valuated graph $G=(V,E)$ where the set of vertices $V=\cro{n}$ and there is an oriented edge between each vertices $\alpha, \beta \in V$ of value $w_{\alpha \beta}$. 
 
\begin{theorem} \label{graphs}
Let $i,j \in \cro{n}$.
Suppose that there exists a path in $G$ with finite valuation between the vertices $i,j \in V$. 
Then $N-e_i+e_j \in \sum_k F_k$. 
Moreover, there are no negative cycles and there is a shortest path between $i$ and $j$. 
Let $(\alpha_u)_{0 \leq u \leq p}$ be any sequence such that $\alpha_0=i$, $\alpha_p=j$ and let $\alpha_0 \rightarrow \alpha_1 \dots \alpha_{p-1} \rightarrow \alpha_p$ be a shortest path between $i$ and $j$. 
Let also $(k_u)_{0 \leq u \leq p-1}$ be any sequence such that $w^{k_u}_{\alpha_u \alpha_{u+1}} = w_{\alpha_u \alpha_{u+1}}$ for all $0 \leq u \leq p-1$. 
Let us finally define the vectors $v_k^*$, $k \in \cro{K}$ such that $v_{k_u}^* = u^*_{k_u} -e_{\alpha_u} + e_{\alpha_{u+1}}$ for each $0 \leq u \leq p-1$ and $v_k^*=u_k^*$ for each $k \notin \{ k_0, \dots, k_{p-1} \}$. 
Then, $\psi(N-e_i+e_j) = \sum_k \langle \rho_k,v_k^* \rangle$.

\end{theorem}

\begin{proof}
By Lemma~\ref{lemma-feasible} and~\ref{lemma-infconv}, we know that $N-e_{i}+e_{j} \in \sum_{k}{F_{k}}$ if and only if there exists $K$ vectors $v_k^*$ such that  $N-e_{i}+e_{j}=\sum_{k}{v_{k}^{*}}$ with each $v_{k}^{*} \in F_{k}$ and $ \psi(N-e_{i}+e_{j})=-\sum_{k}{\langle \rho_{k},v_{k}^{*} \rangle}$. We consider $\psi(N)=-\sum_{k}{\langle \rho_{k},u_{k}^* \rangle}$ with each $u_{k}^* \in F_{k}$. Hence, $\psi(N-e_{i}+e_{j})-\psi(N)$ is equal to:
\[
\min_{v_{k} \in F_{k} \text{ and } \sum_{k}{v_{k}=N-e_{i}+e_{j}}} \sum_{k}{\langle \rho_{k},u_{k}^*-v_{k} \rangle}.
\]

We have $\sum_k (u_k^*-v_k)=e_{i}-e_{j}$.
When $v_{k}$ describes $F_{k}$, the possible $u_{k}^*-v_{k}$ are the vectors $x_k$ with the following properties:
\begin{align*}
& \sum_{\alpha=1}^n x_k(\alpha) = 0,  \quad \forall \alpha \text{ s.t. } u_k^*(\alpha)=1, x_k(\alpha) \in \{0;1\}, \\
& \forall \alpha \in \mathcal{J}_k, x_k(\alpha)=0, \quad \forall \alpha \text{ s.t. } u_k^*(\alpha)=0, x_k(\alpha) \in \{-1;0\}
\end{align*}
Hence, $\psi(N-e_i+e_j) - \psi(N) = \sum_k \langle \rho_k,x_k^* \rangle $, where $x_k^*$ is such that $\# \{ \alpha \mid x_k^*(\alpha)=1 \} =  \# \{ \alpha \mid x_k^*(\alpha)=-1 \}$. 
Consequently, $\psi(N-e_i+e_j) - \psi(N)$ can be written as a sum of $w^{k}_{\alpha\beta}$ for certain $\alpha, \beta$. Because of the condition $\sum_{k}{u_{k}^*-v_{k}}=e_{i}-e_{j}$, we have $\psi(N-e_i+e_j) - \psi(N)=w_{\alpha_0\alpha_1}^{k_0}+w_{\alpha_1\alpha_2}^{k_1}+ \dots +w_{\alpha_{p-1}\alpha_p}^{k_{p-1}}$, 
with the notations introduced in Theorem~\ref{graphs}.

Consider now the graph defined in Theorem~\ref{graphs}.
If there exists a path between $i$ and $j$, then its value can be written $w_{\beta_0\beta_1}^{l_0}+w_{\beta_1\beta_2}^{l_1}+ \dots +w_{\beta_{q-1}\beta_q}^{l_{q-1}}$ (with the convention $\beta_0=i$ and $\beta_q=j$).
By defining $v_k=u_k^*$ if $k \notin \{ l_0,\dots,l_{q-1} \}$ and $v_{l_u} = u_{l_u}^* - e_{\beta_u} + e_{\beta_{u+1}}$ for $0 \leq u \leq q-1$ , the value of the path is equal to $\sum_k \langle \rho_k,u_k^*-v_k \rangle$. 
Because $w_{\beta_u \beta_{u+1}}^{l_u} < + \infty$, we have $u^*_{l_u}(\beta_u)=1$ and $u^*_{l_u}(\beta_{u+1})=0$. 
Then, each $v_k \in F_k$.
Consequently, the value $\min_{v_{k} \in F_{k} \text{ and } \sum_{k}{v_{k}=N-e_{i}+e_{j}}} \sum_{k}{\langle \rho_{k},u_{k}^*-v_{k} \rangle}$ is finite and $N-e_i+e_j \in \sum_k F_k$.
Moreover, the value $\psi(N-e_{i}+e_{j})-\psi(N)$ corresponds to the minimal values of the path between $i$ and $j$ in $G$, that is the shortest path.
Hence, if the value of the shortest path is $\sum_{u=0}^{p-1} w^{l_u}_{\alpha_u \alpha_{u+1}}$, we have $\psi(N-e_i+e_j)-\psi(N)= \sum_k \langle \rho_k, u_k^*-v_k^* \rangle$, with $v_k^*$ defined as in the statement of Theorem~\ref{graphs}. 
Moreover, we can prove that there exists no cycle with negative weight in this graph.
Suppose that such a cycle exists.
It can be written $w_{\gamma_{0}\gamma_{1}}^{l_{0}}+w_{\gamma_{1}\gamma_{2}}^{l_1}+ \dots + w_{\gamma_{r}\gamma_{0}}^{l_r}<0$.
For all $i \in \{0 \dots r\}$, we have $u_{l_i}(\gamma_i)=1$ and $u_{l_i}(\gamma_{i+1})=0$.
We consider for $k \in \cro{K}$ the vectors $v_k$ defined by $v_{l_i} = u^*_{l_i} - e_{\gamma_i} + e_{\gamma_{i+1}}$, and $v_{k}=u^*_{k}$ for $ k \notin \{ l_0, \dots, l_r \}$.
We have $\sum_{k}{u^*_{k}-v_{k}}=0$ and so $\sum_{k}{\langle \rho_{k},u_{k} \rangle} = \sum_{k}{\langle \rho_{k},v_{k} \rangle} + w_{\alpha_{1}\alpha_{2}}^{k_{1}}+w_{\alpha_{2}\alpha_{3}}^{k_{2}}+ \dots + w_{\alpha_{p}\alpha_{1}}^{k_{p}} < \sum_{k}{\langle \rho_{k},v_{k} \rangle} $ which refutes the optimality of the vectors $u^*_{k}$ in the definition of $\psi(N)$.    
\end{proof} 

\begin{example} We consider the cell (a) of Figure~\ref{fig-five}. 
We build the graph associated to $N=(3,3,1)$ (see Figure~\ref{graph-example}).

\begin{figure}[h]
\begin{center}
\begin{tikzpicture} [scale=1.5]
	\draw (-1,0) circle (0.3cm);
	\node at (-1,0) {1};
	\draw (0.5,1) circle (0.3cm);
	\node at (0.5,1) {2};
	\draw (0.5,-1) circle (0.3cm);
	\node at (0.5,-1) {3};
	\draw[->, >=latex] (-0.75,0.3) -- (0.2,0.9);
	\node at (-0.5,0.75) {0};
	\draw[-> , >=latex] (0.3,0.7) -- (-0.65,0.1);
	\node at (-0.2,0.2) {1.5};
	\draw[-> , >=latex] (0.6,0.65) -- (0.6,-0.65);
	\node at (0.8,0) {0.5};
	\draw[-> , >=latex] (0.4,-0.65) -- (0.4,0.65);
	\node at (0.25,0) {1};
	\draw[-> , >=latex] (-0.7,-0.2) -- (0.2,-0.8);
	\node at (-0.4,-0.75) {0}; 
\end{tikzpicture}
\end{center}

\caption{Graph $G$ associated to the vector $N=(3,3,1)$}
\label{graph-example}
\end{figure}
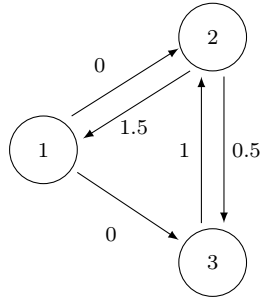

Consider $N'=N-e_1+e_2=(2,4,1)$. The shortest path in $G$ is $1 \rightarrow 2$ with $w_{12}=0=w_{12}^1$. 
Then, according to Theorem~\ref{graphs}, the optimal decomposition of $(2,4,1)$ is $v_1^*=(0,1,0)$, $v_2^*=(1,0,1)$, $v_3^*=(0,1,0)$, $v_4^*=(1,1,0)$ and $v_5^*=(0,1,0)$.
\end{example}

Thanks to Theorem~\ref{graphs}, if we know that a vector $N$ belongs to $\sum_k F_k$, 
it is possible to check whether a vector $N-e_{i}+e_{j}$ belongs to  $F_{k}$  by checking if there exists a path between $i$ and $j$ in the graph $G=(V,E)$. 
Generally, $G$ has $n$ vertices and $n^2$ edges. 
From each vertex $i \in V$, it is possible to find if there exists a path between $i$ and $j$ by using a depth-first or breadth first search algorithm in $O(n^2)$ operations. 
Consequently, the number of operations needed to evaluate $g$ is $O(n^3)$. 

According to Theorem~\ref{graphs}, by checking if $N-e_i+e_j \in B$, we obtain the optimal decomposition of $N-e_i+e_j = \sum_k v_k^*$ such that $\psi(N-e_i+e_j) = - \sum_k \langle \rho_k,v_k^* \rangle$ by solving a shortest path problem between two vertices.
This can be done in $O(n^3)$ operations thanks to Ford-Bellman algorithm (\cite{bellman1958routing}, \cite{ford1956network}), because the graph $G$ has $n$ vertices and at most $n^2$ edges.
Hence, according to Theorem~\ref{decomposition}, it suffices to solve the bilevel problem~\ref{bilevel-telecom} to solve the linear feasibility problem of Lemma~\ref{linear-feasibility}. 
Moreover, this problem can also be viewed as a shortest path problem in $G$, according to the following result.

\begin{theorem} \label{feasibility-graphs}
Consider $K$ vectors $u_k^* \in F_k$ for each $k \in \cro{K}$ such that, if we define $N = \sum_k u_k^*$, we have $\psi(N)=-\sum_k \langle \rho_k,u_k^* \rangle$.
Consider the graph $G$ associated to $N$. 
Consider an index $s \in \cro{n}$. 
Let $M >0$ be any real scalar such that $M \geq  n \max_{i,j \in \cro{n}} w_{ij}$
and let us modify $G$ such that for all $t \in \cro{n}$ with $t \neq s$ and $w_{st}=+\infty$, we have $w_{st}=M$.
Let us define a vector $y^* \in \R^n$ by $y^*_s=0$ and for each $t \in \cro{n}$ with $t \neq s$, $y^*_t$ is the length of the shortest path between $s$ and $t$ in $G$. 
Then, for $M$ sufficiently large and for each $k \in \cro{K}$, $u_k^* \in \arg\max_{u_k \in F_k} \langle \rho_k+y^*,u_k \rangle$. 
\end{theorem}

\begin{proof}
According to Lemma~\ref{linear-feasibility}, a vector $y \in \R^n$ is such that for every $k \in \cro{K}$, $$u_k^* \in \arg\max_{u_k \in F_k} \langle \rho_k+y,u_k \rangle$$ if and only if the following inequalities are satisfied:
\[
\forall k \in \cro{K}, \; \forall i,j \notin \mathcal{J}_k, \text{ such that } u_k^*(i)=1,u_k^*(j)=0, \;  \rho_k(i)+y_i \geq \rho_k(j)+y_j.
\]
Consider such a vector $y$.
Consider also the graph $G$ associated to $N$
The previous inequalities can be rewritten $\forall k \in \cro{K}, \forall i,j \in \cro{n}, \; y_j - y_i \leq w^k_{ij}$, 
or equivalently : $\forall i,j \in \cro{n}, \; y_j - y_i \leq w_{ij}$. 
For each $\delta \in \R$, $y + \delta e$ is also a solution. 
Consequently, it is possible to fix a coordinate to $0$.
Take a coordinate $s$ such that $y_s = 0$. 
Consider $M >0$ such that $M \geq n \max_{i,j} w_{ij}$ and modify the graph $G$ as in the statement of the theorem. 
Consider an elementary cycle (that is a cycle containing no smaller cycle) of the modified graph. 
The cycle has no more than $n-1$ edges. 
Suppose that exactly $q$ edges have a modified weight, with $0 \leq q \leq n-1$.  
If $q=0$, then no edge has a modified weight, and this cycle is a cycle of $G$. 
So, its weight is nonnegative.
If $q \geq 1$, then the total weight of the cycle is bigger than $qM + (n-1-q) \min_{i,j} w_{ij} \geq n (\max_{i,j} w_{ij} - \min_{i,j} w_{ij}) \geq 0$.
Consequently, the modified graph has no negative cycles.

For each $t \in \cro{n}$, with $t \neq s$, there exists a path between $s$ and $t$. 
Let us define $y^*$ such that $y^*_s=0$ and for each $t \in \cro{n}$ with $t \neq s$, $y^*_t$ corresponds to the length of the shortest path between $s$ and $t$. 
Consider $i,j \in \cro{n}$. 
Then $y^*_i + w_{ij}$ is the length of a path between $s$ and $j$ defined as the concatenation of the shortest path between $s$ and $i$ and the edge $i \rightarrow j$. 
So $y^*_i + w_{ij} \geq y^*_j$. 
Hence, according to Lemma~\ref{linear-feasibility}, we have for each $k \in \cro{K}$, $u_k^* \in \arg\max_{u_k \in F_k} \langle \rho_k+y^*,u_k^* \rangle $.  
\end{proof}

These different results lead to Algorithm~\ref{algo-bilevel} to solve the bilevel problem~\ref{bilevel-telecom}.
First, we have to find an initial point $N$ in $\sum_k F_k$,
with its optimal decomposition $\sum_k u_k^*$. 
We can calculate for each $k \in \cro{K}$ and for each $i,j \notin \mathcal{J}_k$ the value $w^k_{ij}$, store them, 
and then define the graph $G$ associated to $N$. 
Hence, with a graph search algorithm, we know for each $i,j \in \cro{n}$ whether $N-e_i+e_j \in \sum_k F_k$ or not,
and can calculate $g(N-e_k+e_k)$ for each $k,l \in \cro{n}$ and find $i,j \in \argmax_{k,l} g(N-e_k+e_l)$.
By finding the shortest path between $i$ and $j$ in $G$, we obtain the optimal decomposition $N-e_i+e_j = \sum_k v_k^*$.
Like in Algorithm~\ref{algo1}, if $g(N-e_i+e_j) \leq g(N)$, then $N^*=N$ is the maximum value of $g$ over $\sum_k F_k$.
Else, we take $N:= N-e_i+e_j$.
For all the indices $k$ such that $u_k^* \neq v_k^*$, we evaluate the new value of $w_{ij}^k$ and we define the graph $G$ associated to $N-e_i+e_j$ and restart the algorithm.
Notice that the number of indices $k$ such that $u_k^* \neq v_k^*$ is bounded by the length of the shortest path in $G$; it means that this number is less than $n$. 
After finding the optimal $N^*$ and having its optimal decomposition $N^*=\sum_k u_k^*$, we can redefine the graph associated to $N^*$ and return an optimal $y^*$ defines as in the statement of Theorem~\ref{feasibility-graphs}.

Algorithm~\ref{algo-bilevel} can be written as follows.
We take in input a function GraphSearch, which associate to a graph $G$ (defined by the weight vector $w$ of its edges) a Boolean vector $b$ such that $b_{ij}=1$ if there is an edge between $i$ and $j$ and $0$ otherwise. 
We also take a function ShortestPath, which associate to a graph $G$ (also defined by the weight vector $w$) and two vertices $i$ and $j$, the value $v$ of the shortest path and a vector $path$ with the indices of this shortest path.
Finally, we consider the function ShortestPath2, which associate to $w$ and a vertex $s$ a vector corresponding to the values of the shortest path between $s$ and all other vertices in $G$.
For much ease, we denote by $f^*$ the function $f^*:N \mapsto f(N)+M\sum_{i=1}^n \max(N_i-N^C_i,0)$.

\begin{algorithm}
\small
\begin{algorithmic}
\REQUIRE $u_k^* \in F_k, \forall k \in \cro{K}$,$\rho_k, \forall k \in \cro{K}$, $f^*$, $GraphSearch$, $ShortestPath$, $ShortestPath2$, $s \in \cro{n}$
\ENSURE $N^*$ optimal number of customers, $y^*$ optimal discount vector
\STATE $N \leftarrow \sum_{k=1}^K u_k^*$
\FORALL{$k \in \cro{K}$}
\FORALL{$i,j \notin \mathcal{J}_k$}
\IF{$u_k(i)=1$ \AND $u_k(j)=0$}
\STATE $w_{ij}^k \leftarrow \rho_k(i)-\rho_k(j)$
\ENDIF
\ENDFOR
\ENDFOR
\FORALL{$i,j \in \cro{n}$}
\STATE $w_{ij} \leftarrow \min_{k \in \cro{K}} w^{k}_{ij}$;
\qquad $k_{ij} \in \arg\min_{k \in \cro{K}} w^{k}_{ij}$
\ENDFOR
\STATE $stop \leftarrow 0$
\WHILE{$stop = 0$}
\STATE $b \leftarrow GraphSearch(w)$
\STATE $g_N \leftarrow f^*(N)$;
\quad $g^* \leftarrow \max_{u,v \in \cro{n}, b_{uv}=1} f^*(N-e_u+e_v)$;
\;  $i,j \in \arg\max_{u,v \in \cro{n}, b_{uv}=1} f^*(N-e_u+e_v)$
\IF{$g^* \leq g_N$}
\STATE $stop \leftarrow 1$
\ELSE
\STATE $(v,path) \leftarrow Shortestpath(w,i,j)$; 
\qquad $N \leftarrow N-e_i+e_j$
\FOR{$q=1$ \TO $Length(path)-1$}
\STATE $\alpha \leftarrow path(q)$;
\qquad $\beta \leftarrow path(q+1)$;
\qquad $k \leftarrow k_{\alpha\beta}$;
\qquad $u_k^*(\alpha)=0$;
\qquad $u_k^*(\beta)=1$
\FORALL{$\gamma \notin \mathcal{J}_k$}
\STATE $w^k_{\alpha \gamma} \leftarrow +\infty$;
\qquad $w^k_{\gamma \beta} \leftarrow +\infty$
\IF{$u_k^*(\gamma)=1$}
\STATE $w^k_{\gamma \alpha} \leftarrow \rho_k(\gamma)-\rho_k(\alpha)$
\ELSE
\STATE $w^k_{\beta \gamma} \leftarrow \rho_k(\beta)-\rho_k(\gamma)$
\ENDIF
\ENDFOR
\ENDFOR
\FORALL{$i,j \in \cro{n}$}
\STATE $w_{ij} \leftarrow \min_{k \in \cro{K}} w^{k}_{ij}$;
\qquad $k_{ij} \in \arg\min_{k \in \cro{K}} w^{k}_{ij}$
\ENDFOR
\ENDIF
\ENDWHILE
\STATE $M \leftarrow 1 + n \max_{i,j \in \cro{n}} w_{ij}$
\FORALL{$t\in \cro{n}$}
\IF{$t \neq s$ AND $w_{st}=+\infty$}
\STATE $w_{st}=M$
\ENDIF
\STATE $y^* \leftarrow Shortestpath2(w,s)$
\STATE $y^*_s \leftarrow 0$
\ENDFOR
\end{algorithmic}
\caption{Solving the bi-level problem, for one application and one type of contract} \label{algo-bilevel}
\end{algorithm}

Note that the pseudo-polynomial time bound for Murota's greedy algorithm~\ref{algo1} given by Proposition~\ref{comp-algogreedy} leads in this special case to a polynomial time bound, 
as explained in the following result.

\begin{theorem} \label{complexity-algo-bilevel}
Let us define $R = \sum_k R_k$, for each $k \in \cro{K}$ $n_k = n - \# \mathcal{J}_k$ (that is the number of possible non-zero entries of the vectors of $F_k)$ and $\overline{R} = \sum_k R_k(n_k-R_k)$.
Algorithm~\ref{algo-bilevel} returns a global optimizer with a time complexity of $O(R(n^3+\overline{R}))$ and a space complexity of $O(\overline{R})$.
\end{theorem}

\begin{proof} The vector returned by the algorithm is a global optimizer according to Algorithm~\ref{algo1} and Theorem~\ref{graphs}. 
The initialization consists in taking vectors in each $F_k$ and in adding them; it can be done in $O(K)$ operations. 
Then, to define the graph $G$, we have to calculate $w_{ij}^k$ for each $i,j \notin \mathcal{J}_k$ and each $k \in \cro{K}$, 
and to store the values.
Let us define for each $k \in \cro{K}$ $n_k=n-\# \mathcal{J}_k$.
For each $k \in \cro{K}$, we have $R_k \leq n_k$, and there are precisely $R_k$ coordinates of $u_k^*$ equal to 1 for each $u_k^* \in F_k$. 
Then, for each $k \in \cro{K}$, there are exactly $R_k(n_k-R_k)$ finite values of $w_{ij}^k$ to store.
Then, by defining $\overline{R} = \sum_k R_k(n_k-R_k)$, we need $O(\overline{R})$ operations to define $w_{ij}$ and $k_{ij}$.
The function $GraphSearch$ needs $O(n^3)$ operations by a depth-first or breadth-first algorithm to know if there is a path between $i$ and $j$.
The function $ShortestPath$ needs also $O(n^3)$ operations to calculate the shortest path between $i$ and $j$ with Ford-Bellman algorithm. 
The length of the path is bounded by $n$. 
Consequently, there is less than $n$ vectors $u_k^*$ which have to be updated;
and then less than $2nn_k$ values $w_{\alpha\beta}^k$ to update. 
$\overline{R}$ operations are needed to calculate the new values of $w_{ij}$ and $k_{ij}$. 
So, the number of operations in each step of the "while" loop is $O(n^3+n\overline{R})$. 
The number of iterations of the loop is the same as in Algorithm~\ref{algo1}, and is bounded by $K_1$ where $K_1= \max(||x-y||_1, x,y \in \sum_k F_k)$. 
For each $x,y \in \sum_k F_k$, we have:
\[
||x-y||_1 = \sum_{i=1}^n |x_i - y_i| \leq \sum_{i=1}^n (x_i + y_i) = 2 R
\] 
by defining $R=\sum_{k=1}^K R_k$.
Finally, to find the optimal $y^*$, $n^2$ operations are needed to find $M$, and $O(n^3)$ operations are needed to evaluate the function $ShortestPath2$ by using again the Ford-Bellman algorithm.
Step 7 consists in calculating the shortest path between a vertex $s$ and the other ones in a graph with $n$ vertices and $n^2$ edges. 
Then, Step 7 can be obtained in $O(n^3)$ thanks to Ford-Bellman algorithm. 
Hence, the global time complexity of Algorithm~\ref{algo-bilevel} is $O(R(n^3+\overline{R}))$ and space complexity is $O(\overline{R})$. 
\end{proof}  

Notice that for each $k \in \cro{K}$, $n_k \leq n$ and $1 \leq R_k \leq n_k$. 
Then $K \leq R \leq nK$ and $0 \leq \overline{R} \leq Kn^2$. 
Therefore, the time complexity of Algorithm~\ref{algo-bilevel} is $O(Kn^3(K+n))$ in the worst case, whereas the space complexity is $O(Kn^2)$. 

\begin{example}
Consider again Example~\ref{ex-1} together with the concave function $f$ defined by
\[
f: N \mapsto -\sum_{t,l}{N(t,l)^{2}}.
\]
We suppose that $\forall k, \mathcal{J}_{k}=\emptyset$. Hence, we can prove that $\sum_{k}{F_{k}}=\lbrace N \in \N^{3}|\sum_{i=1}^{3}{N_{i}}=7$ and $ \max(N_{i})\leq 5 \rbrace$. First, we want to solve
$
\max_{N \in \sum_{k}{F_{k}}}
{-(N_{1}^{2}+N_{2}^{2}+N_{3}^{2})}$.
We start from $N^{(0)}=(5,2,0)$, a feasible point. 
Following Algorithm~\ref{algo1}, we compute $N^{(1)}=(4,2,1)$ and $N^{(2)}=(3,2,2)$ which is a minimizer. We take $N^{*}=(3,2,2)$. Now, we solve
$\max_{\scriptstyle
u_{1} \in F_{1},\dots,u_{5} \in F_{5}, 
\sum_{k=1}^{5}{u_{k}}=N^{*} 
}
{\sum_k 
{\langle \rho_{k},u_{k} \rangle}}$.
We obtain $u_{1}^{*}=\left[ 1,0,0 \right]$, $u_{2}^{*}=\left[ 1,0,1 \right]$, $u_{3}^{*}=\left[ 0,1,0 \right]$, $u_{4}^{*}=\left[ 1,0,1 \right]$, $u_{5}^{*}=\left[ 0,1,0 \right]$.
Applying Lemma~\ref{linear-feasibility},
we obtain the linear inequalities 
$y_{1}^{*}-y_{2}^{*} \leq 3/2  $,
$0 \leq y_{1}^{*}-y_{3}^{*}$
and
$-1 \leq y_{2}^{*}-y_{3}^{*} \leq -1/2$.
In particular, $y^{*}=(3/4,0,3/4)$ is an optimal solution.
\end{example}

\subsection{A particular case : theory of majorization}

Algorithm~\ref{algo-bilevel} can be accelerated in the particular case $\forall k \in \cro{K}, \; \mathcal{J}_{k}= \emptyset$,
that is $F_{k}=\left\lbrace u_{k} \in \{ 0;1 \}^{n} | \sum_{i=1}^{n} u_{k}(i)=R_{k} \right\rbrace$.

As previously, an important step of the maximization of the function $g$ consists in being able to know whether a point belongs to $\sum_{k}F_{k}$ or not.
In this particular case, we can use the {\em majorization order}~\cite{olkin1979inequalities}. 
For every $x \in \R^{n}$, denote by $x_{[1]} \geq \cdots \geq x_{[n]}$ the coordinates of $x$ arranged in nonincreasing order. A vector $x \in \R^{n}$ is said to be {\em majorized} by another vector $y \in \R^{n}$, denoted $x \prec y$,
if
$\sum_{i=1}^{n}{x_{i}}=\sum_{i=1}^{n}{y_{i}} $
and
$\forall 1 \leq k \leq n-1$, 
$\sum_{i=1}^{k}{x_{[i]}} \leq  \sum_{i=1}^{k}{y_{[i]}}$.

We have the following result.
\begin{theorem}[Gale-Ryser , see \protect{\cite[Th.\ 7.C.1]{olkin1979inequalities}}]
Let $a \in \N^k$ and $b \in \N^n$ be two integer vectors with nonnegative values. 
Let $a^* \in \N^n$ defined by $a^*_i = \# {j \mid a_j \geq i}$. Then, the following assertions are equivalent:
\begin{enumerate}
\item $b \prec a^*$
\item There exists a matrix $U \in \mathcal{k,n}(\Z)$ such that for each $i,j$, $u_{ij} \in \{0;1 \}$, $\forall 1 \leq i \leq k, \; \sum_{j=1}^n u_{ij}=a_i$ and $\forall 1 \leq j \leq n, \; \sum_{i=1}^k u_{ij}=b_j$
\end{enumerate}
\label{galeryser}
\end{theorem}

\begin{corollary}
Denoting by $f_{r}=(1,\dots,1,0,\dots,0)$ the vector with exactly $r$ $1$ and 
by $p_{r}= \# \lbrace k | R_{k}=r \rbrace$, for $1 \leq r \leq n$, we have
$\sum_{k}{F_{k}} = \lbrace N \in \N^{n} | N \prec \sum_{r=1}^{n}{p_{r}f_{r}} \rbrace$.
\label{cor-galeryser}
\end{corollary}

\begin{proof}
A vector $N$ belongs to $\sum_k F_k$ if and only if for each $i \in \cro{n}$, $N_i$ corresponds to the sum of the coefficients of the $i$-th column of a matrix of size $K \times n$ with coefficients in $\{ 0;1 \}$ and such that the sum of the coefficients of the $k$-th line is $R_k$.
We conclude by~\ref{galeryser}.
\end{proof}

\begin{example}
Consider Example~\ref{ex-1}.
 We have $p_{1}=3$, $p_{2}=2$ and $p_{3}=0$. So $N$ is feasible iff $N$ verifies $N \prec (5,2,0)$ .
\end{example}

Like for Algorithm~\ref{algo-bilevel}, we need to know for a given $N \in \sum_k F_k$ whether $N-e_i+e_j \in \sum_k F_k$ for each $i,j \in \cro{n}$. 
It is possible to answer to this question in polynomial time in $n$ by sorting $N-e_i+e_j$ for each $i,j$ and by checking the condition $N-e_i+e_j \prec N^{\max}$. 
The time complexity of such a procedure is $O(n^3 \log(n))$. 
However, it can be accelerated thanks to the following result.

\begin{lemma} \label{neighbour-major}
Let $N \in \sum_k F_k$, and $i,j \in \cro{n}$. 
Let $S$ be the function defined on $\R^n \times \cro{n}$ such that $\forall x \in \R^n, \forall k \in \cro{n}$, $S(x,k)$ is the sum of the $k$ largest values of the coordinates of $x$.
Suppose finally that $N_j$ is the $k_j$-th largest value of the coordinates of $N$ (if $k_j>1$, then we suppose that the $k_j-1$-th largest value of $N$ is strictly bigger than $N_j$),
and that $N_i$ is the $k_i$-th largest value of the coordinates of $N$ (if $k_i<n$, then we suppose that the $k_i+1$-th largest value of $N$ is strictly smaller than $N_j$).
Then $N-e_i+e_j \in \sum_k F_k$ if and only if $N_i > 0$ and, either $N_i > N_j$ or $\forall k_j \leq k \leq k_i, S(N,k) < S(N^{\max},k)$.
\end{lemma}

\begin{proof}
Suppose $N-e_i+e_j \in \sum_k F_k$. 
Then $N_i - 1 \geq 0$ and $N_i > 0$.
Moreover, suppose $N_i \leq N_j$. Then, $N_i - 1 < N_j+1$ and $S(N,k)=S(N,k)+1$. Then, $S(N,k) < S(N,k)+1 = S(N^{\max},k)$.

Conversely, if $N_i > 0$, then all the coordinates of $N-e_i+e_j$ are nonnegative integers. 
If $N_i > N_j$, then we easily see that $N-e_i+e_j \prec N$. 
So $N-e_i+e_j \prec N^{\max}$ and $N-e_i+e_j \in \sum_k F_k$. 
Suppose that $N_i \leq N_j$.
Because we suppose that the $k-1$-th largest value of $N$ is strictly bigger than $N_j$, then $k_i > k_j$. We also suppose that $\forall k_j \leq k \leq k_i, S(N,k) < S(N^{\max},k)$.
The $k-1$-th largest value of $N$ is strictly bigger than $N_j$, so it is bigger than $N_j+1$. 
Consequently, we have for all $1 \leq l \leq k-1$, $S(N-e_i+e_j,l) = S(N,l) \leq S(N^{\max},l)$ (because $N \prec N^{\max}$). 
Moreover, $\forall k_j \leq k \leq k_i-1, S(N,k) < S(N^{\max},k)$. 
Because the $k_i+1$-th larger coordinate of $N$ is strictly smaller than $N_i$, then it is smaller than $N_i+1$ and we have $S(N-e_+e_j,k_i) = S(N,k_i) \leq S(N^{\max},k_i)$ and $\forall l \geq k_i+1$, $S(N-e_+e_j,l) = S(N,l) \leq S(N^{\max},l)$. 
Hence, $N-e_i+e_j \prec N^{\max}$ and $N-e_i+e_j \in \sum_k F_k$.
\end{proof}

To solve the bilevel problem~\ref{bilevel-telecom} in this specific case,
we need to find $u_1^* \in F_1, \dots, u_K^* \in F_K$ such that $\psi(N^*) = - \sum_k \langle \rho_k,u_k^* \rangle $.
In Algorithm~\ref{algo-bilevel}, such vectors $(u_k)^*$ are found in the same time as $N^*$. 
Then, to accelerate Algorithm~\ref{algo-bilevel}, we need to be able to solve this problem rapidly. 
In particular, to use a classical linear programming approach leads to a $O((Kn)^{3,5})$ time complexity, 
which is not acceptable. 
The problem to solve can be written:
\begin{problem} \label{mincostflow}
\[
\max_{\substack{u_1, \dots, u_K \in \{ 0;1 \}^n \\ \forall k, \; \sum_{i=1}^n u_k(i)= R_k \\ \forall i, \; \sum_{k=1}^K u_k(i) = N_i}} \sum_{k=1}^K \langle \rho_k,u_k \rangle 
\]
\end{problem}
We already mentioned in the proof of Theorem~\ref{galeryser} that the constraints of this linear program can be written $0 \leq u \leq 1, \; Au = b$, where $A$ is a totally unimodular matrix.
Therefore, the value of this problem is equal to the value of its continuous relaxation. 
Moreover, it can be interpreted as a minimum cost flow problem (see~\cite[Ch.\ 12]{schrijver2003combinatorial} for background). 
We define a bipartite graphs with vertices $i \in \cro{n}$ and $k \in \cro{K}$, and edges between each $i \in \cro{n}$ and each $k \in \cro{K}$. 
Each vertex $i \in \cro{n}$ has an incoming flow equal to $N_i$, whereas each vertex $k \in \cro{K}$ has an outgoing flow equal to $R_k$. 
Moreover, the capacity of each edge is $1$, 
meaning that each flow $u_k(i)$ satisfies $0 \leq u_k(i) \leq 1$, and a cost $-\rho_k(i)$ is associated to each edge. 
Hence, the problem consists in finding the flow $u$ minimizing the total cost in this graph. 
Plenty of algorithms exist to solve such a problem. 
In our case, we have $K \gg n$. 
According to Theorem~\ref{complexity-algo-bilevel},
Algorithm~\ref{algo-bilevel} needs $O(Rn^2(K+n))$ operations to solve Problem~\ref{bilevel-telecom}.
Notice that $K \leq R \leq nK$.
Therefore, in order to accelerate Algorithm~\ref{algo-bilevel} in the studied case,
we need an algorithm solving the flow problem with a complexity depending on $K$ in $K^\alpha$ with $\alpha <2$.

We can interpret the minimum cost flow problem as a minimum cost circulation problem, as presented in~\cite[Ch.\ 12]{schrijver2003combinatorial}. 
We introduce a sink $t$. 
We define an edge between each $k \in \cro{K}$ and $t$ of cost equal to $0$, with a lower-bound for the flow equal to $R_k$ and a capacity of $R_k$.
We also define an edge between $t$ and each $i \in \cro{n}$ of cost equal to $0$, with a lower-bound for the flow equal to $N^*_i$ and a capacity of $N^*_i$.
Such a graph is represented on Figure~\ref{fig-circulation}.

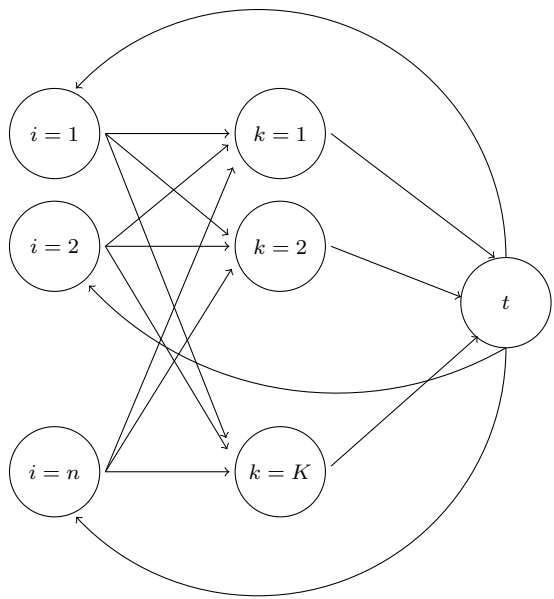
\begin{figure}[h]
\begin{center}
\begin{tikzpicture} [scale=1.5]
	\draw (-2,1.5) circle (0.4cm);
	\node at (-2,1.5) {$i=1$};
	\draw (-2,0.5) circle (0.4cm);
	\node at (-2,0.5) {$i=2$};
	\draw (-2,-1.5) circle (0.4cm);
	\node at (-2,-1.5) {$i=n$};
	\draw (0,1.5) circle (0.4cm);
	\node at (0,1.5) {$k=1$};
	\draw (0,0.5) circle (0.4cm);
	\node at (0,0.5) {$k=2$};
	\draw (0,-1.5) circle (0.4cm);
	\node at (0,-1.5) {$k=K$};
	\draw (2,0) circle (0.4cm);
	\node at (2,0) {$t$};
	\draw[->] (-1.55,1.5)--(-0.45,1.5);
	\draw[->] (-1.55,1.5)--(-0.46,0.6);
	\draw[->] (-1.55,1.5)--(-0.48,-1.2);
	\draw[->] (-1.55,0.5)--(-0.45,0.5);
	\draw[->] (-1.55,0.5)--(-0.46,1.4);
	\draw[->] (-1.55,0.5)--(-0.47,-1.3);
	\draw[->] (-1.55,-1.5)--(-0.45,-1.5);
	\draw[->] (-1.55,-1.5)--(-0.43,0.3);
	\draw[->] (-1.55,-1.5)--(-0.42,1.2);
	\draw[->] (0.45,1.5)--(1.9,0.4);
	\draw[->] (0.45,0.5)--(1.6,0.05);
	\draw[->] (0.45,-1.45)--(1.75,-0.3);
	\draw[->] (2,0.4) arc (0:137:2.2cm);
	\draw[->] (2,-0.4) arc (-60:-137:3cm);
	\draw[->] (2,-0.4) arc (0:-137:2.2cm);
\end{tikzpicture}
\end{center}
\caption{Minimum cost flow problem transformed in a minimum cost circulation problem. 
The flow in the edges between each $i$ and $k$ is in $[0,1]$, the flow in the edges between each $k$ and $t$ is equal to $R_k$, and the flow in the edges between $t$ and each $i$ is $N^*_i$.}
\label{fig-circulation}
\end{figure}

Such a graph has $|V| = K + n +1$ vertices and $|E| = Kn + K + n $ edges. 
The sum of the capacities of the different edges is $2R + Kn$. 
In~\cite[Sec.\ 3.3]{gabow1989faster}, an algorithm is proposed to solve such a problem.
Different complexity bounds of such an algorithm are given in~\cite[Th.\ 3.5]{gabow1989faster}. 
In the case $K \gg n$, the optimal vectors $u_1^*,\dots,u_K^*$ can be found in $O((Kn)^{3/2} \log((K+n)||\rho||_{\infty}))$.

We can now write an algorithm for solving the bilevel problem in this specific case. 
We need first to calculate $N^{\max}=\sum_{r=0}^n p_r f_r$, where $p_r$ is defined as in the statement of Theorem~\ref{galeryser},
and to find an initial point $N \in \sum_k F_k$.
We apply the same method as in Algorithm~\ref{algo1}. 
In order to calculate $g(N-e_i+e_j)$ for each $i,j \in \cro{n}$,
we sort the coordinate of $N$ in the decreasing order, 
and we use Lemma~\ref{neighbour-major} to decide whether $N-e_i+e_j \in F_k$ for all $i,j$. 
We use the same loop as in Algorithm~\ref{algo1} to compute an $N^*$ such that $g(N^*)$ is the maximal value of $g$ over $\sum_k F_k$. 
Then, we solve the minimum cost flow problem~\ref{mincostflow}, as described previously, to find the optimal $u_k^*$ and then we use Lemma~\ref{feasibility-graphs} to determine an optimal $y^*$. 
It leads to Algorithm~\ref{algo-major}. 
The function $Sort$ associates to a vector $x \in \R^n$ a couple $(y,ind)$, where $y$ is a permutation of $x$ such that $y_1 \geq \dots \geq y_n$ and $ind$ is such that $x_i = y_{ind(i)}$ for each $i \in \cro{n}$.
The function $S$ is defined by $S(x,k) = \sum_{i=1}^n x_i$.
The function $MinCostFlow$ associates to the different vectors $(\rho_k)_{k \in \cro{K}}$ the vectors $(u_k^*)_{k \in \cro{K}}$ solving the minimum cost flow problem~\ref{mincostflow}.
The functions $f^*$ and $ShortestPath2$ are defined as for Algorithm~\ref{algo-bilevel}.

\begin{algorithm} \footnotesize
\begin{algorithmic}
\REQUIRE $N \in \sum_k F_k$, $N^{\max}$, $\rho_k, \forall k \in \cro{K}$, $f^*$, $S$, $Sort$, $MinCostFlow$, $ShortestPath2$, $s \in \cro{n}$
\ENSURE $N^*$ optimal number of customers, $y^*$ optimal discount vector 
\STATE $s^{\max}_1 \leftarrow 0$
\FOR{$i=1$ \TO $n$}
\STATE $s^{\max}_i \leftarrow s^{\max}_i+N^{\max}_i$
\ENDFOR
\STATE $stop \leftarrow 0$
\WHILE{$stop = 0$}
\STATE $(N^{sort},ind) \leftarrow Sort(N)$;
\quad $s_1 = 0$
\FOR{$i=1$ \TO $n$}
\STATE $s_i \leftarrow s_i+N^{sort}(i)$
\FORALL{$j \in \cro{n}$}
\IF{$N(i)=0$}
\STATE $b(i,j) \leftarrow 0$
\ELSE
\STATE $b(i,j) \leftarrow 1$
\ENDIF
\ENDFOR
\IF{$s_i = s^{\max}_i$}
\FOR{$j=1$ \TO $i$}
\STATE $b(ind(i),ind(j)) \leftarrow 0$
\ENDFOR
\FOR{$j = i$ \TO $n$}
\STATE $b(ind(j),ind(i)) \leftarrow 0$
\ENDFOR
\ENDIF
\ENDFOR
\STATE $g_N \leftarrow f^*(N)$;
\quad $g^* \leftarrow \max_{u,v \in \cro{n}} b(u,v)f^*(N-e_u+e_v)$;
\;  $i,j \in \arg\max_{u,v \in \cro{n}} b(u,v)f^*(N-e_u+e_v)$
\IF{$g^* \leq g_N$} 
\STATE $stop \leftarrow 1$
\ENDIF
\ENDWHILE
\STATE $(u_1^*,\dots,u_K^*) \leftarrow MinCostFlow((\rho_k)_{k \in \cro{K}})$
\FORALL{$k \in \cro{K}$}
\FORALL{$i,j \notin \mathcal{J}_k$}
\IF{$u_k^*(i)=1$ \AND $u_k^*(j)=0$}
\STATE $w_{ij}^k \leftarrow \rho_k(i)-\rho_k(j)$
\ENDIF
\ENDFOR
\ENDFOR
\FORALL{$i,j \in \cro{n}$}
\STATE $w_{ij} \leftarrow \min_{k \in \cro{K}} w^{k}_{ij}$;
\qquad $k_{ij} \in \arg\min_{k \in \cro{K}} w^{k}_{ij}$
\ENDFOR
\STATE $M \leftarrow 1 + n \max_{i,j \in \cro{n}} w_{ij}$
\FORALL{$t\in \cro{n}$}
\IF{$t \neq s$ AND $w_{st}=+\infty$}
\STATE $w_{st}=M$
\ENDIF
\STATE $y^* \leftarrow Shortestpath2(w,s)$
\STATE $y^*_s \leftarrow 0$
\ENDFOR
\end{algorithmic}
\caption{Solving the bilevel problem, in the case of majorization} \label{algo-major}
\end{algorithm} 

\begin{theorem} \label{complexity-algo-major}
Let us define $||\rho||_{\infty} = \max_{k \in \cro{K}, i \in \cro{n}} |\rho_k(i)|$, $R = \sum_k R_k$, for each $k \in \cro{K}$ $n_k = n - \# \mathcal{J}_k$ and $\overline{R} = \sum_k R_k(n_k-R_k)$.
Algorithm~\ref{algo-major} is correct and returns a global optimizer in $O(Rn^2 +(Kn)^{3/2} \log((K+n)||\rho||_{\infty})+\overline{R} + n^3)$ time and $O(Kn +n^2)$ space.
\end{theorem}

\begin{proof} 
According to Theorem~\ref{decomposition}, Theorem~\ref{galeryser}, Lemma~\ref{neighbour-major} and Algorithm~\ref{algo1}, 
this algorithm returns an optimal solution $N^*$ of the high-level problem and an optimal discount vector $y^*$. 
Similarly as in the proof of Algorithm~\ref{algo-bilevel}, the number of calls of the "while" loop is bounded by $R$. 
The function $Sort$ needs $O(n \log(n))$ time and space operations. 
$O(n^2)$ operations are needed to evaluate the vector $b$, 
then the global time complexity of the "while" loop is $O(Rn^2)$ whereas the space complexity is $O(n^2)$. 
Then, the optimal vectors $u_1^*,\dots,u_K^*$ can be obtained in $O((Kn)^{3/2} \log((K+n)||\rho||_{\infty}))$ time and $O(Kn)$ space. 
By calculating only the finite values of $w_{ij}^k$ (which are not necessary stored here), the number of operations needed to determine each $w_{ij}$ and $k_{ij}$ is $O(\overline{R})$, with $\overline{R} = \sum_k R_k (n_k - R_k)$ and for each $k \in \cro{K}$, $n_k = n_\# \mathcal{J}_k$.
We need only $O(n^2)$ space to store the values $w_{ij}$ and $k_{ij}$.  
Finally, the vector $y^*$ can be found by using the Ford-Bellman algorithm in a graph of $n$ vertices and $n^2$ edges, that is in time complexity of $O(n^3)$. 
\end{proof}

In the worst case, we have $R=Kn$ and $\overline{R}=Kn^2$. 
Then, the time complexity of Algorithm~\ref{algo-major} is $O(Kn^3 +(Kn)^{3/2} \log((K+n)||\rho||_{\infty}))$
If the number of bits needed to write $||\rho||_{\infty}$ is polynomial in $n$ and if $K \gg n$, then Algorithm~\ref{algo-major} is faster than Algorithm~\ref{algo-bilevel}. 
We finally notice that a minimum cost flow problem is strongly polynomial time solvable, 
and it is then possible to adapt Algorithm~\ref{algo-major} to return an optimal $y^*$ in strongly polynomial time. 
However, Algorithm~\ref{algo-major} does not go faster than Algorithm~\ref{algo-bilevel} in this case.

\section{The general algorithm}\label{sec-general}

In this section, we come back to the general bilevel problem \ref{highlevel} proposed in Section~\ref{sec-model}, and extend Algorithm~\ref{algo-bilevel} to it.
In the low level problem of each customer, the consumptions for different contents verify the constraints $\forall a \in \cro{A}, \sum_{t=1}^{T}{u_{k}^{a}(t)}=R_{k}^{a}$, $\forall t \in \mathcal{I}_{k}^{a}, a \in \cro{A}, u_{k}^{a}(t)=0$ and $ \forall t \in \cro{T}, \sum_{a \in \cro{A}}{u_{k}^{a}(t)} \leq 1$. 
We make the assumption that for each customer $k$, the sets of possible instants 
at which this customer makes a request for the different applications are disjoint,
meaning that for any two applications $a\neq a'$, the complements of $\mathcal{I}_{k}^{a}$ 
and $\mathcal{I}_{k}^{a'}$ in $\cro{T}$ have an empty intersection.
Then the constraint $ \forall t \in \cro{T}, \sum_{a \in \cro{A}}{u_{k}^{a}(t)} \leq 1$ is automatically verified and the low-level problem of each customer can be separated into different optimization problems corresponding to the consumption vector $u_{k}^{a}$ of each customer $k$ for each application $a$.
Each of these problems takes the following form:
\begin{problem}\label{lowlevelappli}
\begin{equation} 
\max_{u_{k}^{a}\in \{0,1\}^T}  \sum_{t=1}^{T}{\left[ \rho_{k}^{a}(t)+\alpha_{k}^{a} y^{a,b}(t,L_{t}^{k}) \right] u_{k}^{a}(t)}
\end{equation}
\[
\text{s.t. } \sum_{t=1}^{T}{u_{k}^{a}(t)}=R_{k}^{a}, \quad \enspace 
\forall t \in \mathcal{I}_{k}^{a}, a \in \cro{A}, u_{k}^{a}(t)=0\enspace .
\]
\end{problem}

We denote by $F_{k}^{a}$ the feasible set of this problem. 
The above assumption (that the 
complements of $\mathcal{I}_{k}^{a}$ 
and $\mathcal{I}_{k}^{a'}$ have an empty intersection)
is relevant in particular if only one kind of application is sensitive to price incentives.
For instance, requests for downloading data can be anticipated (see \cite{tadrous2013pricing}) and it makes sense to assume that customers are only sensitive to incentives for this kind of contents. In this case, the assumption means that customers wanting to download data can shift their consumption only at instants when they do not request another kind of content. 

Moreover, under this assumption, the decomposition theorem is still valid and Problem~\ref{highlevel} can be solved with the following method:

\begin{theorem}[Decomposition (general case)] \label{decomposition-general}
The bilevel problem~\ref{highlevel} can be solved as follows:
\begin{enumerate}
\item Find an optimal solution $(N^{a,b})^*$ to the high level problem with unknown $N^{a,b}$ for each $a \in \cro{A}$, $b \in \cro{B}$:
\begin{problem} \label{high-level-general}
\begin{align*} 
& \max_{N^{a,b} \in \sum_{k} F_k^a}  \sum_{t,l}{ \left( \sum_{a\in \cro{A}}{ \sum_{b \in \cro{B}} {\gamma_{b}N^{a,b}(t,l)s_{l}^{a,b}(N(t,l))}} \right)}  \\
& \text{ s.t. } \forall t,l, \;  N(t,l) = \sum_{a \in \cro{A}} \sum_{b \in \cro{B}} N^{a,b} (t,l) \quad \text{and} \quad \forall t,l, \; N(t,l) \leq N_l^C \enspace .
\end{align*}
\end{problem}
\item For each $a \in \cro{A}$ and $b \in \cro{B}$, find vectors $((u_k^a)^*)_{k \in \mathcal{K}_b}$ solutions of the following problem:
\[
\max_{\substack{(u_k^a \in F_k^a)_{k \in \mathcal{K}_b} \\ \sum_{k \in \mathcal{K}_b} u_k^a = (N^{a,b})^*}} 
\sum_{k \in \mathcal{K}_b}{\langle \rho_{k}^a,u_{k}^a \rangle}  \enspace .
\]
\item Find for each $a \in \cro{A}$ and $b \in \cro{B}$ a vector $y_{a,b}^{*}$ such that $\forall k \in \mathcal{K}_b$, $$(u_k^a)^* \in \arg\max_{u_k^a \in F_k^a} \langle \rho_k^a,u_k^a \rangle$$. 
\end{enumerate}
\end{theorem}

\begin{proof}
The different problems corresponding for each $a \in \cro{A}$, for each $b \in \cro{B}$ and for each $k \in \mathcal{K}_b$ to Problem~\ref{lowlevelappli} are independent. 
Thus, according to Lemma~\ref{lemma-feasible}, the global bilevel program consists in solving Problem \ref{high-level-general}. 
Moreover, the optimal decomposition of $(N^{a,b})^*$ and the optimal price vector $(y^{a,b})^*$ are totally independent for each $a \in \cro{A}$ and $b \in \cro{B}$. 
Then, the proof of the last two parts in the theorem is the same as in Theorem~\ref{decomposition}.  
\end{proof}

The last two parts of Theorem~\ref{decomposition-general} are independent for each $a \in \cro{A}$ and $b \in \cro{B}$. 
Thus, they can be solved similarly as in the case of one kind of application and one kind of contracts,
studied in Section~\ref{sec-polytime}.
We need to solve Problem~\ref{high-level-general}. 
The function to optimize is separable (it can be written as a sum of function depending only of one coordinate),
but these functions are not concave in $(N^{1,1},\dots,N^{A,B}) \in \R^{nAB}$. 
However, because each function $s_{l}^{a,b}$ is concave nonincreasing and each $N^{a,b}(t,l)$ is positive, we notice that $\forall a' \in \cro{A}, b' \in \cro{B}$,
the function which sends $N^{a',b'}(t,l)$ to 
$\sum_{a\in \cro{A}}{ \sum_{b \in \cro{B}} {\gamma_{b}N^{a,b}(t,l)s_{l}^{a,b}(N(t,l))}}$ is still concave. Consequently, the function to optimize in Problem~\ref{high-level-general} is $M$-concave in each vector $N^{a,b}\in \mathbb{Z}^{T\times L}$ considered separately,
the other one being fixed.  
This leads to a block descent method, 
in which we use the same scheme as in Algorithm~\ref{algo1}, successively, 
to maximize the objective function
over every vector $N^{a,b}$.
We denote by $f(N^{1,1}, \dots , N^{A,B})$ the objective function of the high-level problem.
We consider for each $a,b$ a vector $N^{a,b} \in \sum_{k \in \mathcal{K}_b} F_k^a$. 
For each couple $(a,b)$ taken successively, we find $(i^{a,b},j^{a,b})$ belonging to:
\[
\underset{{(k,l) \; \text{s.t.} \; N^{a,b} - e_k +e_l \in \sum_{k \in \mathcal{K}_b} F_k^a}}{\arg\max} f(N^{1,1}, \dots, N^{a,b} -e_k+e_l, \dots, N^{A,B})
\]
If $f(N^{1,1} - e_{i^{1,1}} + e_{j^{1,1}}, \dots, N^{A,B} - e_{i^{A,B}} + e_{j^{A,B}}) \leq f(N^{1,1},\dots,N^{A,B})$, 
then the algorithm stops and returns $(N^{1,1},\dots,N^{A,B})$. 
Otherwise, we take for each $a,b$, $N^{a,b} := N^{a,b}-e_{i^{a,b}} + e_{j^{a,b}}$ and begin again. 
Consequently, Algorithm~\ref{algo-bilevel} can be modified to solve the bilevel problem~\ref{high-level} in the general case. 
It leads to Algorithm~\ref{algo-general}.
The function $GraphSearch$, $ShortestPath$ and $ShortestPath2$ are the same as for Algorithm~\ref{algo-bilevel}. 
The function $f^*$ is here defined by:
\[
f^*:(N^{1,1}, \dots, N^{A,B}) \mapsto \sum_t \sum_l \left[ \sum_{a\in \cro{A}}{ \sum_{b \in \cro{B}} {\gamma_{b}N^{a,b}(t,l)s_{l}^{a,b}(N(t,l))}} -M \max(N(t,l)-N^C_l,0) \right]
\]
with $N(t,l) = \sum_{a \in \cro{A}} \sum_{b \in \cro{B}} N^{a,b}(t,l)$. 

\begin{algorithm}
\footnotesize
\caption{Solving the bilevel problem for an arbitrary number of types of contracts.} \label{algo-general} 
\begin{algorithmic}
\REQUIRE $(u_k^a)^* \in F_k^a, \forall a \in \cro{A},  \forall k \in \cro{K}$, $\rho_k^a, \forall a \in \cro{A}, \forall k \in \cro{K}$, $f^*$, $GraphSearch$, $ShortestPath$, $ShortestPath2$, $s \in \cro{n}$
\ENSURE $N^*$ optimal number of customers, $y^*$ optimal discount vector
\FORALL{$a \in \cro{A}, b \in \cro{B}$}
\STATE $N^{a,b} \leftarrow \sum_{k \in \mathcal{K}_b} (u_k^a)^*$
\FORALL{$k \in \mathcal{K}_b$}
\FORALL{$t,t' \notin \mathcal{I}_k^a$}
\STATE $i=(t,L_k(t)) \quad j=(t',L_k(t'))$, 
\IF{$u_k^a(i)=1$ \AND $u_k^a(j)=0$}
\STATE $w_{ij}^{k,a} \leftarrow \rho_k^a(i)-\rho_k^a(j)$
\ENDIF
\ENDFOR
\ENDFOR
\FORALL{$t,t' \in \cro{T}$, $l,l' \in \cro{L}$}
\STATE $i=(t,l) \quad j=(t',l')$
\STATE $w_{ij}^{a,b} \leftarrow \min_{k \in \mathcal{K}_b} w^{k,a}_{ij}$;
\qquad $k_{ij}^{a,b} \in \arg\min_{k \in \mathcal{K}_b} w^{k,a}_{ij}$
\ENDFOR
\ENDFOR
\STATE $stop \leftarrow 0$
\WHILE{$stop = 0$}
\STATE $g_N \leftarrow f^*(N^{1,1},\dots,N^{A,B})$;
\FORALL {$a \in \cro{A}, b \in \cro{B}$}
\STATE $c^{a,b} \leftarrow GraphSearch(w^{a,b})$
\STATE $g^* \leftarrow \max_{u,v \in \cro{T} \times \cro{L}} c^{a,b}_{uv}f^*(N^{1,1},\dots,N^{a,b}-e_u+e_v,\dots,N^{A,B})$; \quad
 $i^{a,b},j^{a,b} \in \arg\max_{u,v \in \cro{T} \times \cro{L}} c^{a,b}_{uv}f^*(N^{1,1},\dots,N^{a,b}-e_u+e_v,\dots,N^{A,B})$
\ENDFOR
\IF{$g^* \leq g_N$}
\STATE $stop \leftarrow 1$
\ELSE
\FORALL{$a \in \cro{A}, b \in \cro{B}$}
\STATE $(v,path) \leftarrow Shortestpath(w^{a,b},i,j)$; 
\qquad $N^{a,b} \leftarrow N^{a,b}-e_{i^{a,b}}+e_{j^{a,b}}$
\FOR{$q=1$ \TO $Length(path)-1$}
\STATE $\alpha \leftarrow path(q)$;
\qquad $\beta \leftarrow path(q+1)$;
\qquad $k \leftarrow k^{a,b}_{\alpha\beta}$;
\qquad $(u_k^a)^*(\alpha)=0$;
\qquad $(u_k^a)^*(\beta)=1$
\FORALL{$\gamma \notin \mathcal{J}_k^a$}
\STATE $w^{k,a}_{\alpha \gamma} \leftarrow +\infty$;
\qquad $w^{k,a}_{\gamma \beta} \leftarrow +\infty$
\IF{$(u_k^a)^*(\gamma)=1$}
\STATE $w^{k,a}_{\gamma \alpha} \leftarrow \rho_k^a(\gamma)-\rho_k^a(\alpha)$
\ELSE
\STATE $w^{k,a}_{\beta \gamma} \leftarrow \rho_k^a(\beta)-\rho_k^a(\gamma)$
\ENDIF
\ENDFOR
\ENDFOR
\FORALL{$t,t' \in \cro{T}, l,l' \in \cro{L}$}
\STATE $i = (t,l)$ \quad $j=(t',l')$
\STATE $w^{a,b}_{ij} \leftarrow \min_{k \in \mathcal{K}_b} w^{k,a}_{ij}$;
\qquad $k^{a,b}_{ij} \in \arg\min_{k \in \mathcal{K}_b} w^{k,a}_{ij}$
\ENDFOR
\ENDFOR
\ENDIF
\ENDWHILE
\STATE $M \leftarrow 1 + ABTL \max_{i,j \in \cro{T} \times \cro{L}, a \in \cro{A}, b\in \cro{B}} w_{ij}^{a,b}$
\FORALL{$t\in \cro{n}$}
\IF{$t \neq s$ AND $w_{st}=+\infty$}
\STATE $w_{st}=M$
\ENDIF
\STATE $y^* \leftarrow Shortestpath2(w,s)$
\STATE $y^*_s \leftarrow 0$
\ENDFOR
\end{algorithmic}
\end{algorithm}

Because the objective function of Problem~\ref{high-level-general} is not $M$-convex in $(N^{1,1},\dots,N^{A,B})$, we have no guarantee of convergence of Algorithm~\ref{algo-general} to a global optimal of the function $f^*$. 
However, we can characterize the nature of the optimum returned by Algorithm~\ref{algo-general}.
In order to estimate the complexity of Algorithm~\ref{algo-general}, we define the function $\Delta f^*$ by:
\begin{align*}
\Delta f^*(& N^{1,1},\dots,N^{A,B}) = - f^*(N^{1,1},\dots,N^{A,B}) \\
+ & \max_{\substack{u^{a,b},v^{a,b} \in \cro{T} \times \cro{L} \\ N^{a,b}-e_{u^{a,b}} + e_{v^{a,b}} \in \sum_{k \in \mathcal{K}_b} F_k^a}} f^*(N^{1,1}-e_{u^{1,1}} +e_{v^{1,1}},\dots,N^{A,B}-e_{u^{A,B}} +e_{v^{A,B}})
\end{align*}

If for each $a,b$ we have $u^{a,b}=v^{a,b}$, then $\Delta f^*(N^{1,1},\dots,N^{A,B}) = 0$. 
Thus, we have $$\Delta f^*(N^{1,1},\dots,N^{A,B}) \geq 0$$. 
Because the set $\prod_{a,b} (\sum_{k \in \mathcal{K}_b} F_k^a)$ is finite, we can define the value $\delta g$ by:
\[
\delta g = \min_{\substack{N^{a,b} \in \sum_{k \in \mathcal{K}_b} F_k^a \\  \Delta f^*(N^{1,1},\dots,N^{A,B}) > 0}} \Delta f^*(N^{1,1},\dots,N^{A,B})
\] 
because $f^*$ has not a constant value. 

\begin{theorem} \label{complexity-algo-general}
Let us define $\gamma_{\max} = \max_{b \in \cro{B}} \gamma_b$. 
Let us also define $R = \sum_{a \in \cro{A}} \sum_{k \in \cro{K}} R_k^a$, for each $a \in \cro{A}$ and $k \in \cro{K}$ $n_k^a = TL - \# \mathcal{J}_k^a$ (that is the number of possible non-zero coordinates of the vectors of $F_k^a)$ and $\overline{R} = \sum_a \sum_k R_k^a(n_k^a-R_k^a)$.
Algorithm~\ref{algo-general} terminates in $O(\frac{\gamma_{\max} R}{\delta g}(AB(TL)^3+\overline{R}))$ time and $O(\overline{R})$ space,
and returns vectors $(y^{a,b})^*_{a \in \cro{A}, b \in \cro{B}}$ and $(N^{a,b})^*_{a \in \cro{A}, b \in \cro{B}}$
such that $\forall a \in \cro{A}, b \in \cro{B}, \; \forall N^{a,b} \in \sum_k{\mathcal{K}_b} F_k^a$ : 
\[
 f^*((N^{1,1})^*,\dots, (N^{a,b})^*,\dots,(N^{A,B})^*) \geq f^*((N^{1,1})^*,\dots, N^{a,b},\dots,(N^{A,B})^*)
\]
\end{theorem}

\begin{proof}
Algorithm~\ref{algo-general} continues while the value $g^*$ is strictly larger than $g_N$. 
Because the set $\prod_{a,b} (\sum_{k \in \mathcal{K}_b} F_k^a)$ is finite, the algorithm terminates. 
When it stops, the vector $(N^{a,b})_{a \in \cro{A}, b \in \cro{B}}$ is such that $\forall a \in \cro{A}, b \in \cro{B}, \; \forall u,v \in \cro{T} \times \cro{L}$:
\[
f(N^{1,1},\dots,N^{a,b}-e_u+e_v,\dots,N^{A,B}) \leq f(N^{1,1},\dots,N^{a,b},\dots,N^{A,B}) 
\] 
For each $a,b$, the function $N^{a,b} \mapsto f(N^{1,1},\dots,N^{a,b},\dots,N^{A,B})$ is $M$-concave.
The statement of the theorem comes straightforwardly from the equivalence between local and global optimality for $M$-concave functions.

Algorithm~\ref{algo-general} differs from Algorithm~\ref{algo-bilevel} by the different applications and kind of contracts and by the number of iterations of the loop. 
The set $\cro{K}$ of customers is split following the different kind of contracts $b \in \cro{B}$. 
Thus, we have to define the parameters $w_{ij}^{k,a}$ for each $k \in \cro{K}$ and $a \in \cro{A}$ and the global space complexity becomes $\sum_a \sum_k R_k^a(n_k^a - R_k^a) = \overline{R}$. 
The number of iterations of the loop can be estimated with a pseudo-polynomial bound. 
The algorithm continues while $g^* > g_N$. 
Then, the new value of $g^*$ is $f^*(N^{1,1}-e_{i^{1,1}} +e_{j^{1,1}}, \dots, N^{A,B}-e_{i^{A,B}} +e_{j^{A,B}}$. 
Consequently, at each iteration of the loop, the value of $g^*$ increases of at least $\delta g$ until the algorithm stops. 
The finite values of $f^*$ are nonnegative, and an upper bound is $(\max_{b \in \mathcal{B}} \gamma_b)(\sum_{a}\sum_{k \in \cro{K}} R_k^a) = \gamma_{\max} R$ because each function $s_l$ takes values between $0$ and $1$. 
In each loop, the number of operations is $O(\overline{R} + AB (TL)^3)$ to calculate the new values of $w_{ij}^{a,b}$ and to solve a shortest path problem for each $a$ and $b$ in the graph $G^{a,b}$ with nodes corresponding to all couples in $\cro{T} \times \cro{L}$ and edges with values $w^{a,b}_{ij}$ between vertices $i, j \in \cro{T} \times \cro{L}$. 
\end{proof}

\section{Experimental results} \label{sec-exp}

We consider an application based on real data provided by Orange. 
It involves the data consumptions in an area of $L=43$ cells, during one day divided in time slots of one hour, that is $T=24$ time slots. 
We will focus here our study on price incentives only for download contents.
During this day, a number $K$ of more than $2500$ customers make some requests for downloading data in this area and we are interested in balancing the number of active customers in the network.  
Even though they are insensitive to price incentives, other kind of requests (web, mail, etc.) have to be satisfied and they are
taken into account in the high level optimization problem. 
We consider two classes of users: standard and premium customers. The premium ones demand a better quality of service. Hence, they are less satisfied than the standard customers if they share their cell with  a given number of active customers.  We therefore define the satisfaction function as in Section \ref{sec-model}. The provider wants to favor the premium customers. Hence, we take $\gamma_{b}=2$ for the latter ones and $\gamma_{b}=1$ for the standard customers,
in the high-level optimization problem. We also assume that the premium customers are less sensitive to the incentives, and thus take $\alpha_{k}^{a}=1/2$ for all standard customers and $\alpha_{k}^{a}=1$ for all premium customers in the low-level problem~\ref{lowlevel}.
We estimate very simply the parameters $\rho_{k}$. We take $\rho_{k}(t)=1$ when the customer $k$ consumes download at time $t$ without incentives,
$\rho_{k}(t)=0$ when he does not make any request without incentives but makes a request for download at times $t-1$ or $t+1$ 
(we assume he could shift his consumption of one hour)
and $\rho_{k}(t)=- \infty$ otherwise. 

\begin{figure}[h] 
\begin{center}
\begin{tabular}{cc}
\includegraphics[scale=0.4]{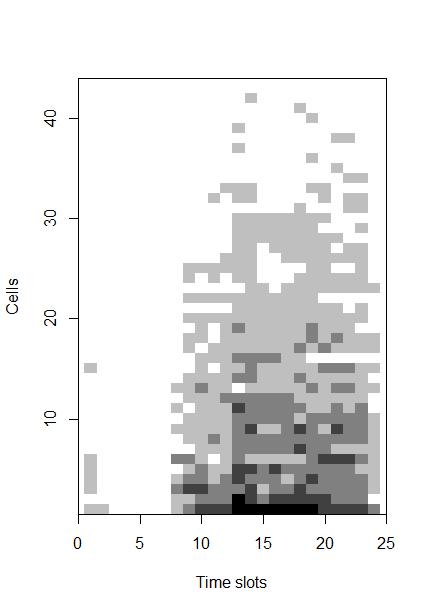}
\includegraphics[scale=0.4]{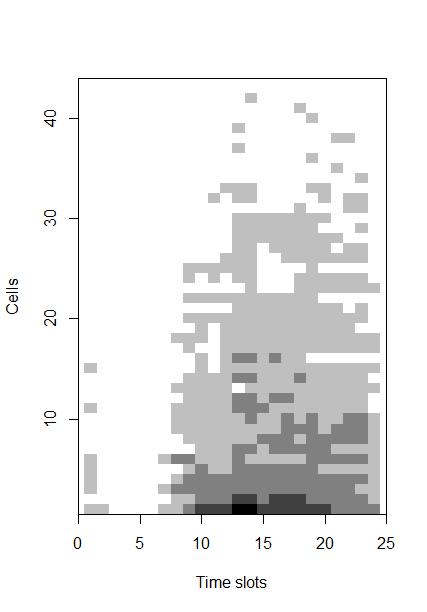}
\end{tabular}
\caption{Satisfaction of premium customers for streaming without (left) and with (right) incentives. The grey level indicates the satisfaction: critical unsatisfaction, $s<0.3$ (black), $0.3<s<0.7$ (dark grey), $0.7<s<0.9$ (grey), $0.9<s<0.99$ (light grey) and complete satisfaction $0.99<s$ (white).}
\label{fig-streaming-premium}
\end{center}
\end{figure}

\begin{figure}[h]
\begin{center}
\begin{tabular}{cc}
\includegraphics[scale=0.4]{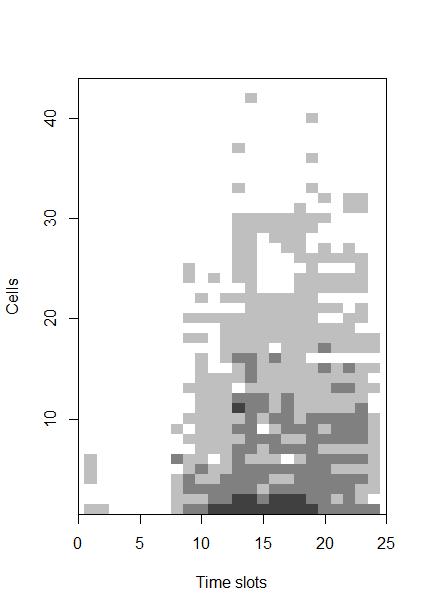}
\includegraphics[scale=0.4]{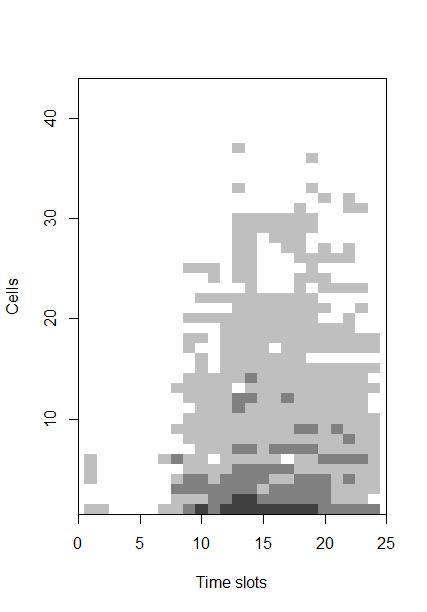}
\end{tabular}
\caption{Satisfaction of standard customers for streaming without (left) and with (right) incentives}
 \label{fig-streaming}
\end{center}
\end{figure}

\begin{figure}[h]
\begin{center} 
\begin{tabular}{cc}
\includegraphics[scale=0.4]{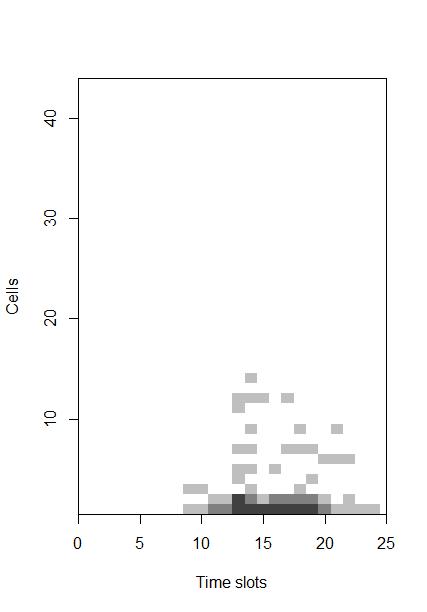}
\includegraphics[scale=0.4]{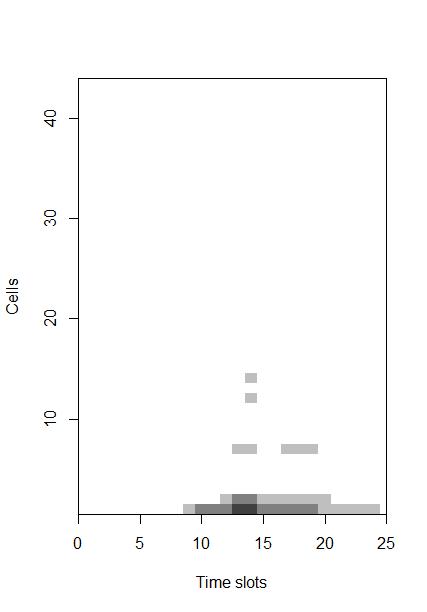}
\end{tabular}
\caption{Satisfaction of premium customers for web, mail or download without (left) and with (right) incentives}
\label{fig-wmd-premium}
\end{center}
\end{figure}

\begin{figure}[h]
\begin{center} 
\begin{tabular}{cc}
\includegraphics[scale=0.4]{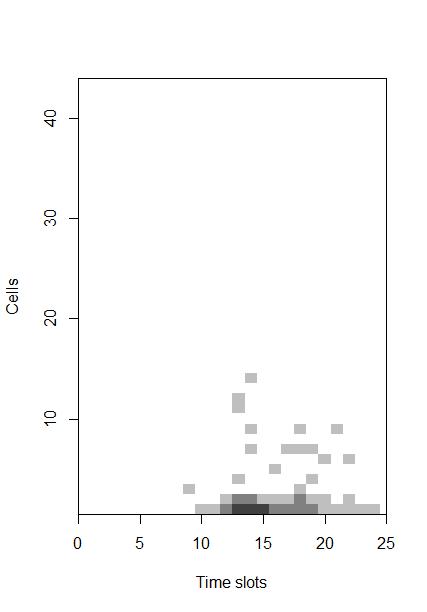}
\includegraphics[scale=0.4]{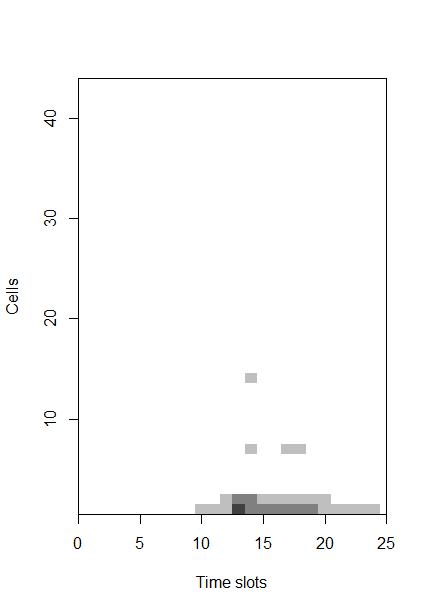}
\end{tabular}
\caption{Satisfaction of standard customers for web, mail or download without (left) and with (right) incentives}
\label{fig-wmd}
\vspace{-.5em}
\end{center}
\end{figure}

\begin{figure}[h]
\begin{center}
\begin{tabular}{cc}
\includegraphics[scale=0.4]{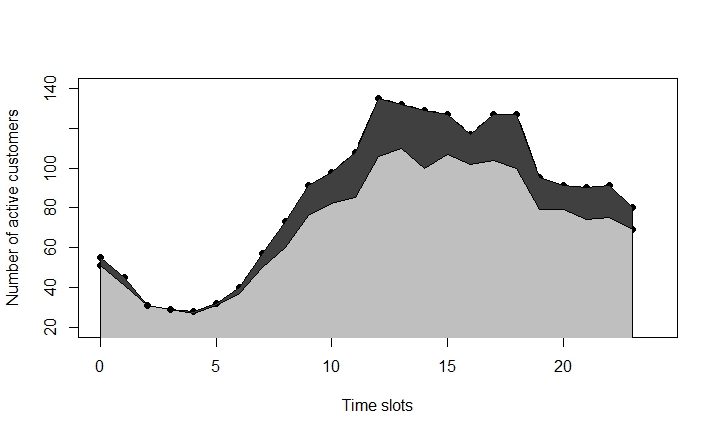}
\includegraphics[scale=0.4]{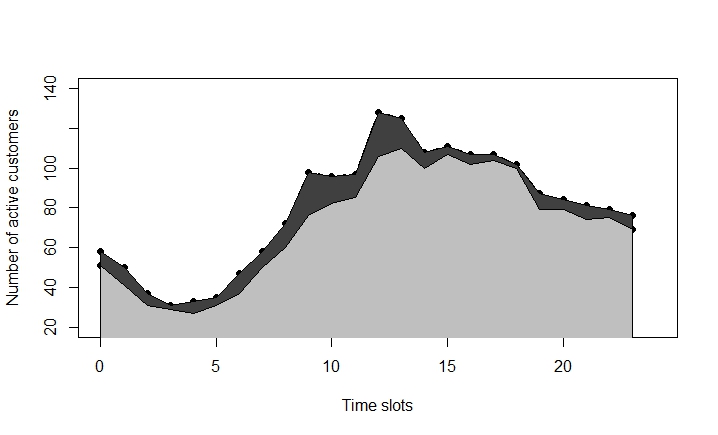}
\end{tabular}
\caption{Traffic in the most loaded cell. The light grey part represents the web, mail and streaming customers who have no incentives and are fixed. The dark grey part corresponds to the download customers in the cell without (left) and with (right) incentives}
 \label{traffic-most-loaded-cell}
\end{center}
\end{figure}

We solve the bilevel problem using Algorithm~\ref{algo-general},
implemented in {\tt Scilab}.
The computation took 9526 seconds on a single core of an Intel i5-4690 processor @ 3.5 GHz.  

On Figures~\ref{fig-streaming-premium}--
\ref{fig-wmd}, we show the evolution of the satisfaction of different kind of customers for different kind of contents without and with incentives. These results show that price incentives have an effective
influence on the load, especially in the most loaded cells (the number of black regions in the space-time coordinates, in which the unsatisfaction of the users is critical, is considerably reduced).
Moreover, 
Figure~\ref{traffic-most-loaded-cell} reveals that the consumption of users is not only moved in time, but also in space: not only some consumption is moved
from the peak hour to the night (off peak), but the surface of the dark grey
region, representing the total download consumption in the cell
over the whole day, is decreased, indicating that some part of the consumption
has been shifted to other cells. 

\section{Conclusion}

We presented here a bilevel model for price incentives in data mobile networks. 
We solved this problem by a decomposition method based on discrete convexity and tropical geometry.
We finally applied our results to real data.
In further work, we shall consider more general models: 
unfixed number of requests, nonlinear preferences of the customers, satisfaction functions of the provider taking into account the profit.
Stochastic models shall also be considered in particular to take into account the partial information of the provider about the customers preferences and trajectories.

\section{Acknowledgments}

We thank the reviewers of our earlier work~\cite{eytard2017bilevel}
for their remarks and comments.
We also thank Orange for providing us real data for our experimental results.

\bibliographystyle{spmpsci}      
\bibliography{Article_EURO}   

%
%

\end{document}